\documentclass[12pt,oneside,english]{amsart}
\usepackage[T1]{fontenc}
\usepackage[latin9]{inputenc}
\usepackage{geometry}
\usepackage{mathrsfs}
\usepackage{amstext}
\usepackage{amsthm}
\usepackage{amssymb}

\makeatletter
\numberwithin{equation}{section}
\numberwithin{figure}{section}
\theoremstyle{plain}
\newtheorem{thm}{\protect\theoremname}[section]
\theoremstyle{plain}
\newtheorem{question}[thm]{\protect\questionname}
\theoremstyle{plain}
\newtheorem{lem}[thm]{\protect\lemmaname}
\theoremstyle{plain}
\newtheorem{prop}[thm]{\protect\propositionname}
\theoremstyle{plain}
\newtheorem{cor}[thm]{\protect\corollaryname}
\theoremstyle{remark}
\newtheorem{rem}[thm]{\protect\remarkname}
\theoremstyle{definition}
\newtheorem{defn}[thm]{\protect\definitionname}

\usepackage{amsthm}
\usepackage{amscd}

\makeatother

\usepackage{babel}
\providecommand{\corollaryname}{Corollary}
\providecommand{\definitionname}{Definition}
\providecommand{\lemmaname}{Lemma}
\providecommand{\propositionname}{Proposition}
\providecommand{\questionname}{Question}
\providecommand{\remarkname}{Remark}
\providecommand{\theoremname}{Theorem}

\begin{document}
\global\long\def\R{\mathbb{R}}%

\global\long\def\C{\mathbb{C}}%

\global\long\def\Q{\mathbb{Q}}%

\global\long\def\Z{\mathbb{Z}}%

\global\long\def\P{\mathbb{P}}%

\global\long\def\T{\mathbb{T}}%

\global\long\def\F{\mathbb{F}}%

\global\long\def\bbN{\mathbb{N}}%

\global\long\def\A{\mathrm{A}}%

\global\long\def\H{\mathrm{H}}%

\global\long\def\D{\mathrm{\mathrm{D}}}%

\global\long\def\B{\mathrm{\mathrm{B}}}%

\global\long\def\N{\mathrm{\mathrm{N}}}%

\global\long\def\inf{\mathrm{inf}}%

\global\long\def\sup{\mathrm{sup}}%

\global\long\def\Hom{\mathbb{\mathrm{Hom}}}%

\global\long\def\Ext{\mathbb{\mathbb{\mathrm{Ext}}}}%

\global\long\def\Ker{\mathbb{\mathrm{Ker}}}%

\global\long\def\Gal{\mathrm{Gal}}%

\global\long\def\Aut{\mathrm{Aut}}%

\global\long\def\End{\mathrm{End}}%

\global\long\def\Char{\mathrm{Char}}%

\global\long\def\rank{\mathrm{rank}}%

\global\long\def\deg{\mathrm{deg}}%

\global\long\def\det{\mathrm{det}}%

\global\long\def\Tr{\mathrm{Tr}}%

\global\long\def\Id{\mathrm{Id}}%

\global\long\def\Spec{\mathrm{Spec}}%

\global\long\def\Lie{\mathrm{Lie}}%

\global\long\def\span{\mathrm{span}}%

\global\long\def\sep{\mathrm{sep}}%

\global\long\def\sgn{\mathrm{sgn}}%

\global\long\def\dist{\mathrm{dist}}%

\global\long\def\val{\mathrm{val}}%

\global\long\def\dR{\mathrm{dR}}%

\global\long\def\st{\mathrm{st}}%

\global\long\def\rig{\mathrm{rig}}%

\global\long\def\cris{\mathrm{cris}}%

\global\long\def\Mat{\mathrm{Mat}}%

\global\long\def\cyc{\mathrm{cyc}}%

\global\long\def\et{\mathrm{\acute{e}t}}%

\global\long\def\Frob{\mathrm{Frob}}%

\global\long\def\SL{\mathrm{SL}}%

\global\long\def\GL{\mathrm{GL}}%

\global\long\def\Br{\mathrm{Br}}%

\global\long\def\Ind{\mathrm{Ind}}%

\global\long\def\LT{\mathrm{LT}}%

\global\long\def\Res{\mathrm{Res}}%

\global\long\def\SL{\mathrm{SL}_{2}}%

\global\long\def\Mod{\mathrm{Mod}}%

\global\long\def\Fil{\mathrm{Fil}}%

\global\long\def\la{\mathrm{la}}%

\global\long\def\pa{\mathrm{pa}}%

\global\long\def\Sen{\mathrm{Sen}}%

\global\long\def\dif{\mathrm{dif}}%

\global\long\def\HT{\mathrm{HT}}%

\global\long\def\Kla{K-\mathrm{la}}%

\global\long\def\Kpa{K-\mathrm{pa}}%

\global\long\def\an{\mathrm{an}}%

\global\long\def\grad{\mathrm{\triangledown}}%

\global\long\def\weak{\rightharpoonup}%

\global\long\def\weakstar{\overset{*}{\rightharpoonup}}%

\title{Lubin-Tate theory and overconvergent Hilbert modular forms of low
weight}
\author{Gal Porat}
\begin{abstract}
Let $K$ be a finite extension of $\Q_{p}$ and let $\Gamma$ be the
Galois group of the cyclotomic extension of $K$. Fontaine's theory
gives a classification of $p$-adic representations of $\Gal\left(\overline{K}/K\right)$
in terms of $(\varphi,\Gamma)$-modules. A useful aspect of this classification
is Berger's dictionary which expresses invariants coming from $p$-adic
Hodge theory in terms of these $\left(\varphi,\Gamma\right)$-modules.

In this article, we use the theory of locally analytic vectors to
generalize this dictionary to the setting where $\Gamma$ is the Galois
group of a Lubin-Tate extension of $K$. As an application, we show
that if $F$ is a totally real number field and $v$ is a place of
$F$ lying above $p$, then the $p$-adic representation of $\Gal\left(\overline{F}_{v}/F_{v}\right)$
associated to a finite slope overconvergent Hilbert eigenform which
is $F_{v}$-analytic up to a twist is Lubin-Tate trianguline. Furthermore,
we determine a triangulation in terms of a Hecke eigenvalue at $v$.
This generalizes results in the case $F=\Q$ obtained previously by
Chenevier, Colmez and Kisin.
\end{abstract}

\maketitle
\tableofcontents{}

\section{Introduction}

Let $p$ be a prime number. Kisin showed in \cite{Ki03} that the
$p$-adic representation $\rho_{f}$ of $\Gal\left(\overline{\Q}/\Q\right)$
attached to a finite slope $p$-adic eigenform $f$ has a very special
property: its restriction to $\Gal\left(\overline{\Q}_{p}/\Q_{p}\right)$
always has a crystalline period. Even better, this period is an eigenvector
for crystalline Frobenius, with eigenvalue coinciding with that arising
from the Hecke action of $U_{p}$ on $f$. Consequently, Kisin was
able to verify the Fontaine-Mazur conjecture for these $p$-adic representations.
In a subsequent work \cite{Co08}, Colmez gave a reinterpertation
of Kisin's result to the effect that the $(\varphi,\Gamma)$-module
attached to $\rho_{f}|_{\Gal\left(\overline{\Q}_{p}/\Q_{p}\right)}$
is an extension of $\left(\varphi,\Gamma\right)$-modules of rank
1, where $\Gamma\cong\Z_{p}^{\times}$ is the Galois group of the
cyclotomic extension of $\Q_{p}$. Colmez coined the term ``trianguline''
for the $p$-adic representations satisfying this property, and studied
them in detail in dimension 2. Then in \cite{Co10} he attached to
any 2-dimensional trianguline representation of $\Gal\left(\overline{\Q}_{p}/\Q_{p}\right)$
a unitary Banach representation of $\GL_{2}\left(\Q_{p}\right)$.
By a suitable continuity argument he was able to extend this procedure
to any $2$-dimensional $p$-adic representation of $\Gal\left(\overline{\Q}_{p}/\Q_{p}\right)$,
thereby constructing the $p$-adic Langlands correspondence for $\GL_{2}(\Q_{p})$.
This circle of ideas came to a satisfying conclusion when Emerton
used this correspondence in \cite{Em11} to show that the trianguline
property at $\Gal\left(\overline{\Q}_{p}/\Q_{p}\right)$ characterizes
these 2-dimensional representations of $\Gal\left(\overline{\Q}/\Q\right)$
which are attached to finite slope $p$-adic eigenforms.

In this article, we are concerned with performing the reinterpertation
step of Colmez in an analogous story when $\Q$ is replaced with a
totally real number field $F$. Namely, the $p$-adic representation
$\rho_{f}$ of $\Gal(\overline{F}/F)$ will be attached to a finite
slope $p$-adic Hilbert eigenform $f$, and we would like to show
$\rho_{f}$ is trianguline at a place $v\mid p$. However, when $F_{v}\neq\Q_{p}$,
there is more than one meaning one can attach to the phrase ``$\rho_{f}$
is trianguline at a place $v\mid p$''. On the one hand, there are
\emph{cyclotomic trianguline} $\Gal(\overline{F}_{v}/F_{v})$-representations.
These are the trianguline representations in the sense of Nakamura
in \cite{Na09}; in that setting, $\Gamma$ is the Galois group of
the cyclotomic extension of $F_{v}$ and is isomorphic to an open
subgroup of $\Z_{p}^{\times}$. On the other hand, there are \emph{Lubin-Tate
trianguline} $\Gal(\overline{F}_{v}/F_{v})$-representations in the
sense of Fourquaux and Xie in \cite{FX13}, where $\Gamma=\Gamma_{F_{v}}$
is the Galois group of a Lubin-Tate extension and is isomorphic to
$\mathcal{O}_{F_{v}}^{\times}$. The representation $\rho_{f}|_{\Gal(\overline{F}_{v}/F_{v})}$
has been known to be cyclotomic trianguline for a while by the global
triangulation theory of Kedlaya-Pottharst-Xiao in \cite{KPX14} and
Liu in \cite{Li12}. Although these results have found many applications,
it seems like this notion of cyclotomic triangulinity may not be the
most suitable for applications to a hypothetical $p$-adic Langlands
correspondence for $\GL_{2}(F_{v})$. Rather, the replacement of $\Z_{p}^{\times}$
by $\mathcal{O}_{F_{v}}^{\times}$ seems more natural, which leads
us in this work to focus on the notion of Lubin-Tate triangulinity
of Fourquaux and Xie.

Let us explain what the difficulties are in proving such a result
when $F_{v}\neq\Q_{p}$. The methods of \cite{KPX14} and \cite{Li12}
can still be used to show the existence of a crystalline period which
is a Frobenius eigenvector. However, Colmez's reformulation of this
condition in terms of $\left(\varphi,\Gamma\right)$-module for $F_{v}=\Q_{p}$
relies on Berger's dictionary, which expresses invariants coming from
$p$-adic Hodge theory in terms of these $\left(\varphi,\Gamma\right)$-modules.
This dictionary is only available in the cyclotomic setting. Indeed,
the proof of this dictionary ultimately relies on Sen theory and the
Cherbonnier-Colmez theorem. Unfortunately, a direct attempt to use
these methods breaks down whenever $F_{v}\neq\Q_{p}$, because of
the failure of the Tate-Sen axioms for Lubin-Tate extensions.

Now let $K$ be a finite extension of $\Q_{p}$. Recall that a representation
$V$ of $\Gal(\overline{K}/K)$ with coefficients in $K$ is called
$K$-analytic if $\C_{p}\otimes_{K}^{\tau}V$ is trivial for each
nontrivial embedding $\tau:K\rightarrow\overline{K}$. In \cite{BC16},
Berger and Colmez were able to find a certain generalization of Sen
theory for Lubin-Tate extensions of $K$ using ideas coming from the
theory of $K$-locally analytic vectors. Berger then used this theory
in \cite{Be16} to prove that $K$-analytic representations are overconvergent,
so that we can associate to $V$ a Lubin-Tate $(\varphi_{q},\Gamma_{K})$-module
$\D_{\rig,K}^{\dagger}(V)$ over the (Lubin-Tate) Robba ring $\B_{\rig,K}^{\dagger}$
(see $\mathsection5$). By adapting the original techniques of \cite{Be02}
to the setting of $K$-locally analytic vectors, we are able to extend
Berger's dictionary to Lubin-Tate extensions of $K$. Our first main
result is the following (see Theorem 5.5).\\

\textbf{Theorem A. }\emph{Let $V$ be a $K$-analytic representation
of $G_{K}$. For $*\in\left\{ \Sen,\dif,\dR,\cris,\st\right\} $,
there is a natural isomorphism
\[
\D_{*,K}(V)\cong\D_{*,K}\left(\D_{\rig,K}^{\dagger}(V)\right).
\]
}\\

For the definition of the functors $\D_{*,K}$ we refer to $\mathsection3$.
When $K=\Q_{p}$ they coincide with the usual definitions and the
theorem is already known, but note that in contrast, it is not a-priori
clear how to even define $\D_{\Sen,K}$ and $\D_{\dif,K}$ in the
general case. 

When $V$ is 2-dimensional, one can deduce from Theorem A that the
Lubin-Tate triangulinity of $V$ can be detected from the existence
of a crystalline period which is a Frobenius eigenvector. On the other
hand, techniques going back to the original paper of Colmez show that
the cyclotomic triangulinity of $V$ can also be detected in a similar
way. From this we deduce that the two notions of triangulinity actually
coincide for $K$-analytic representations of dimension 2 (see Theorem
6.8 for a more precise version).\\

\textbf{Theorem B. }\emph{Let $V$ be a 2-dimensional $K$-analytic
representation of $G_{K}$. Then $V$ is Lubin-Tate trianguline if
and only if $V$ is cyclotomic trianguline.}\\

Although we do not pursue this here, our methods show that under some
genericity assumptions the theorem is true for $V$ of arbitrary dimension.
On the other hand, after the completion of this paper, we have been
informed that a previously unpublished result of Léo Poyeton \cite{Po20}
establishes an equivalence of categories between $K$-analytic Lubin-Tate
$(\varphi_{q},\Gamma_{K})$-modules and $K$-analytic $\B$-pairs.
As an immediate consequence, this gives an independent proof of Theorem
B which extends to $V$ of arbitrary dimension. Indeed, the triangulinity
can be checked in terms of $\B$-pairs, and the rank 1 $\B$-pairs
that appear in the triangulation are attached to both rank 1 cyclotomic
$(\varphi,\Gamma)$-modules and to rank 1 Lubin-Tate $(\varphi_{q},\Gamma_{K})$-modules.

As mentioned above, it is known that the local Galois representations
associated to finite slope overconvergent Hilbert eigenforms are cyclotomic
trianguline. It is then natural to use Theorem B to translate this
into a Lubin-Tate triangulinity result, provided that this local representation
is analytic. Furthermore, it is possible to explicitly determine this
triangulation, generalizing previous work of Chenevier and Colmez
(see Theorem 7.4 for a more precise statement). To state the result,
let $\rho_{f}$ be the $p$-adic representation of $\Gal(\overline{F}/F)$
associated to a finite slope overconvergent Hilbert eigenform of weights
$(k,1,...,1)$ at $v$ and determinant $\eta\chi_{\cyc}^{w-1}$ for
some potentially unramified character $\eta$. For the sake of simplifying
the introduction, assume here that the restriction of $\rho_{f}$
to a decomposition group $G_{F_{v}}=\Gal(\overline{F}_{v}/F_{v})$
has coefficients in $F_{v}$ and that $k,w\in\Z$. Removing these
assumptions only requires introducing more notation. To state the
result, choose a uniformizer $\pi_{v}$ of $F_{v}$, write $\chi_{\pi_{v}}$
for the corresponding Lubin-Tate character and let $a_{v}\in F_{v}^{\times}$
be such that $U_{v}f=a_{v}f$.  If $y\in F_{v}^{\times}$, write
$\mu_{y}:F_{v}^{\times}\rightarrow F_{v}^{\times}$ for the character
defined by $\mu_{y}(z)=y^{\val_{\pi_{v}}(z)}$. Let $x:F_{v}^{\times}\rightarrow F_{v}^{\times}$
be the character $x(z)=z$ and $x_{0}:F_{v}^{\times}\rightarrow F_{v}^{\times}$
be the character $x_{0}(z)=x/\mu_{\pi_{v}}$. We say that $f$ is
$F_{v}$-analytic up to a twist if the same holds for $\rho_{f}|_{G_{F_{v}}}$.\\

\textbf{Theorem C. }\emph{Suppose that $f$ is $F_{v}$-analytic up
to a twist. Then it is Lubin-Tate trianguline. If $\D_{\rig,F_{v}}^{\dagger}\left(\rho_{f}|_{G_{F_{v}}}^{\vee}\right)$
is the $\left(\varphi_{q},\Gamma_{F_{v}}\right)$-module over $\B_{\rig,F_{v}}^{\dagger}$
associated to $\rho_{f}|_{G_{F_{v}}}^{\vee}$, a triangulation is
given by the short exact sequence  
\[
0\rightarrow\B_{\rig,F_{v}}^{\dagger}\left(\delta_{1}\right)\rightarrow\D_{\rig,K}^{\dagger}\left(\rho_{f}|_{G_{F_{v}}}^{\vee}\right)\rightarrow\B_{\rig,F_{v}}^{\dagger}\left(\delta_{2}\right)\rightarrow0,
\]
where $\delta_{2}=\delta_{1}^{-1}\det(V)$ and $\delta_{1}:F_{v}^{\times}\rightarrow F_{v}^{\times}$
is a character. Here and $\delta_{1}$ and $\rho_{f}|_{G_{F_{v}}}$
satisfy the following. }
\begin{enumerate}
\item \emph{If $k\notin\Z_{\geq1}$ then $\delta_{1}=\mu_{a_{v}}x_{0}^{\frac{k-1}{2}}\left(\N_{F_{v}/\Q_{p}}\circ x_{0}\right)^{\frac{1-w}{2}}$
and $\rho_{f}|_{G_{F_{v}}}$ is irreducible and not Hodge-Tate.}
\item \emph{If $k\in\Z_{\geq1}$ and $\val_{\pi_{v}}(a_{v})<\frac{k-1}{2}+\frac{w-1}{2}\left[F_{v}:\Q_{p}\right]$,
then $\delta_{1}=\mu_{a_{v}}x_{0}^{\frac{k-1}{2}}\left(\N_{F_{v}/\Q_{p}}\circ x_{0}\right)^{\frac{1-w}{2}}$
and $\rho_{f}|_{G_{F_{v}}}$ is irreducible and potentially semistable.}
\item \emph{If $k\in\Z_{>1}$ and} $\val_{\pi_{v}}(a_{v})=\frac{k-1}{2}+\frac{w-1}{2}\left[F_{v}:\Q_{p}\right]$,
then either
\begin{enumerate}
\item $\delta_{1}=\mu_{a_{v}}x_{0}^{\frac{k-1}{2}}\left(\N_{F_{v}/\Q_{p}}\circ x_{0}\right)^{\frac{1-w}{2}}$
and \emph{$\rho_{f}|_{G_{F_{v}}}$ is reducible, nonsplit and potentially
ordinary, or}
\item $\delta_{1}=x^{1-k}\mu_{a_{v}}x_{0}^{\frac{k-1}{2}}\left(\N_{F_{v}/\Q_{p}}\circ x_{0}\right)^{\frac{1-w}{2}}$
and \emph{$\rho_{f}|_{G_{F_{v}}}$ is a sum of two characters and
potentially crystalline.}
\end{enumerate}
\item \emph{If $k\in\Z_{\geq1}$ and $\val_{\pi_{v}}(a_{v})>\frac{k-1}{2}+\frac{w-1}{2}\left[F_{v}:\Q_{p}\right]$,
then $\delta_{1}=x^{1-k}\mu_{a_{v}}x_{0}^{\frac{k-1}{2}}\left(\N_{F_{v}/\Q_{p}}\circ x_{0}\right)^{\frac{1-w}{2}}$
and $\rho_{f}|_{G_{F_{v}}}$ is irreducible, Hodge-Tate and not potentially
semistable.}\\
\end{enumerate}
The condition on the weights at $v$ to be of the form $(k,1,...,1)$
is necessary but not sufficient for $f$ to be $F_{v}$-analytic up
to a character twist. In fact, the computations of \cite[Proposition 2.10]{Na09}
and \cite[Theorem 0.3]{FX13} suggest that this stronger condition
of $F_{v}$-analyticity cuts out a locus of codimension $[F_{v}:\Q_{p}]-1$
inside the locus of weights $(k,1,...,1)$ at $v$ of the Hilbert
eigenvariety. However, under suitable local-global compatibility conjectures,
it contains all classical points of weights $(k,1,...,1)$. We believe
it might be possible to obtain a version of Theorem C for arbitrary
$f$ if one works with $\left(\varphi_{q},\Gamma_{F_{v}}\right)$-modules
over multivariable Robba rings as in \cite{Be13}.

Finally, we make some further speculations. For simplicity, assume
that $p$ is inert in $F$. The small slope condition $\val_{p}(a_{p})<\frac{k-1}{2}+\frac{w-1}{2}\left[F_{p}:\Q_{p}\right]$
in Theorem C agrees with the optimal bound in partial weight 1 conjectured
in an unpublished note of Breuil (see Proposition 4.3 of \cite{Br10}).
This suggests that $F_{p}$-analytic finite slope $p$-adic Hilbert
eigenforms of weights $(k,1,...,1)$ and $\val_{p}(a_{p})<\frac{k-1}{2}+\frac{w-1}{2}\left[F_{p}:\Q_{p}\right]$
should be classical. If such a classicality criterion were known,
an argument as in Theorem 6.6 of \cite{Ki03} using our Theorem 7.4
would verify the Fontaine-Mazur conjecture for representations which
arise from $F_{v}$-analytic finite slope $p$-adic Hilbert eigenforms.
Conversely, if the Fontaine-Mazur conjecture were known in our context
then Theorem 7.4 would imply such a classicaility criterion. See $\mathsection7.2$
for a more precise discussion.

In another direction, suppose again that $p$ is inert in $F$ and
that $V$ is a $p$-adic representation of $\Gal(\overline{F}/F)$
which is irreducible, totally odd, unramified at almost all primes
and Lubin-Tate trianguline at $p$. Recall again the theorem of Emerton
(Theorem 1.2.4 of \cite{Em11}) which asserts that if $F=\Q$ and
$\overline{V}$ satisfies certain technical conditions then $V$ is
the character twist of the Galois representation attached to an (elliptic)
overconvergent $p$-adic eigenform of finite slope. In light of Theorem
7.4, we ask the following.
\begin{question}
Is $V$ necessarily the character twist of a representation attached
to an $F_{p}$-analytic overconvergent $p$-adic Hilbert eigenform
of finite slope?
\end{question}

\subsection{Structure of the article}

$\mathsection2$ contains preliminaries regarding locally $K$-analytic
vectors. In $\mathsection3$ we define the functors $\D_{*,K}$ for
\emph{$*\in\left\{ \Sen,\dif,\dR,\cris,\st\right\} $} and prove their
basic properties. In $\mathsection4$ we study some big period rings
and their $K$-locally analytic vectors. Theorem A is proved in $\mathsection5$.
This proof involves reinterperting several constructions in $p$-adic
Hodge theory in terms of $K$-locally analytic vectors, as well as
some computations with rather large rings of periods, and so involves
most of the work done in $\mathsection\mathsection2$-5. In $\mathsection6$
we relate this to Lubin-Tate triangulinity and prove Theorem B. Finally,
in $\mathsection7$ we prove Theorem C, concluding with an example.

\subsection{Notations and conventions}

The field $K$ denotes a finite extension of $\Q_{p}$, with ring
of integers $\mathcal{O}_{K}$, uniformizer $\pi$, and residue field
$k$. The field $E$ is a finite extension of $\Q_{p}$ which contains
$K$. It will serve as a field of coefficients for the objects we
consider. The field $K_{0}=\mathrm{W}(k)[1/p]$ is the maximal unramified
subextension of $K$. Let $q=p^{f}$ be the cardinality of the residue
field and $e$ the absolute ramification index of $K$, so that $[K:\Q_{p}]=ef$.
We let $\Sigma_{K}$ denote the set of embeddings of $K$ into $\overline{\Q}_{p}$.

Denote by $G_{K}$ the absolute Galois group of $K$. If $\mathcal{F}$
is the formal Lubin-Tate group associated to $\pi$, then $K_{n}=K(\mathcal{F}[\pi^{n}])$
and $K_{\infty}=\cup_{n\geq1}K_{n}$ are abelian extensions of $K$
which depend only on $\pi$. The Lubin-Tate character $\chi_{\pi}:G_{K}\rightarrow\mathcal{O}_{K}^{\times}$
is the character given by the action of $G_{K}$ on $\mathcal{F}[\pi^{\infty}]$.
It induces an isomorphism of $\Gamma_{K}=\Gal(K_{\infty}/K)$ with
$\mathcal{O}_{K}^{\times}$. Its kernel is $H_{K}=\Gal(\overline{K}/K_{\infty})$,
and $G_{K}/H_{K}=\Gamma_{K}$. The cyclotomic character $\chi_{\mathrm{cyc}}$
of $G_{K}$ satisfes the relation $\mathrm{N}_{K/\Q_{p}}\circ\chi_{\pi}=\chi_{\mathrm{cyc}}\eta$
for an unramified character $\eta$.

An $E$-linear representation $V$ of $\Gal(\overline{K}/K)$ is called
$K$-analytic if $\C_{p}\otimes_{K}^{\tau}V$ is trivial for each
$\tau\in\Sigma_{K}\backslash\left\{ \Id\right\} $.

Finally, all characters and representations appearing in this article
are assumed to be continuous. We normalize the $p$-adic valuation
and $p$-adic logarithm so that $\val_{p}(p)=1$ and $\log(p)=0$.
The Hodge-Tate weight of $\chi_{\cyc}$ is $1$.

\subsection{Acknowledgments}

I am grateful to my advisor Matthew Emerton for suggesting I try to
prove Lubin-Tate triangulinity results for $p$-adic modular forms,
and for his support throughout the project. I would like to thank
Laurent Berger, Ehud de-Shalit, Yulia Kotelnikova, Hao Lee, and Alexander
Petrov for many useful comments, Eric Stubley for useful explanations
about partial weight 1 Hilbert modular forms, and Richard Moy and
Joel Specter for sharing with me their calculations of the coefficients
of the modular form appearing in $\mathsection7.3$. 

\section{Locally $K$-analytic and pro $K$-analytic vectors}

In this section we give reminders on locally analytic and pro analytic
vectors and gather a few results that will be used in $\mathsection3,\mathsection4$
and $\mathsection5$.

\subsection{Locally analytic and pro analytic vectors}

We briefly recall the treatment given in $\mathsection2$ of \cite{Be16}
and in $\mathsection2$ of \cite{BC16}. 

Let $W$ be a Banach $\Q_{p}$-linear representation of $\Gamma_{K}$.
Given an open subgroup $H$ of $\Gamma_{K}$ with coordinates $c_{1},...,c_{d}:H\xrightarrow{\sim}\Z_{p}^{d}$,
we have the subspace $W^{H-\an}$ of $H$-analytic vectors in $W$.
These are the elements $w\in W$ for which there exists a sequence
of vectors $\left\{ w_{k}\right\} _{k\in\bbN^{d}}$ with $w_{k}\rightarrow0$
and $g(w)=\sum_{k\in\bbN^{d}}c_{k}(g)^{k}w_{k}$ for all $g\in H$.
We write $W^{\la}=\cup_{H}W^{H-\an}$ for the subspace of locally
analytic vectors of $W$. If $W$ is a Fréchet space whose topology
is defined by a countable sequence of seminorms, let $W_{i}$ be the
Hausdorff completion of $W$ for the $i$'th norm, so that $W=\underset{\leftarrow}{\lim}W_{i}$
is a projective limit of Banach spaces. We write $W^{\pa}=\underset{\leftarrow}{\lim}W_{i}^{\la}$
for the subspace of pro analytic vectors. 

For $n\gg0$, we have an isomorphism $l:\Gamma_{K_{n}}\rightarrow\pi^{n}\mathcal{O}_{K}$,
given by $g\mapsto\log(\chi_{\pi}(g))$. We have the subspace $W^{\Gamma_{K_{n}}-\an,K-\la}$
of vectors which are $K$-analytic on $\Gamma_{K_{n}}$, i.e. such
that there exists a sequence $\left\{ w_{k}\right\} _{k\in\mathbb{N}}$
with $\pi^{nk}w_{k}\rightarrow0$ and $g(w)=\sum_{k\in\mathbb{N}^{d}}l(g)^{k}w_{k}$
for all $g\in\Gamma_{K_{n}}$. We write $W^{K-\la}=\cup_{n\gg0}W^{\Gamma_{K_{n}}-\an,K-\la}$
for the subspace of $K$-locally analytic vectors of $W$. If $W=\underset{\leftarrow}{\lim}W_{i}$
is a Fréchet space as above, we write $W^{K-\pa}=\underset{\leftarrow}{\lim}W_{i}^{K-\la}$
for the subspace of pro $K$-analytic vectors. We extend the definitions
of locally $K$-analytic vectors and pro $K$-analytic vectors to
LB and LF spaces (i.e. filtered colimits of Banach spaces and Fréchet
spaces) in the obvious way.

For each $\tau\in\Sigma_{K}$, there is a differential operator $\nabla_{\tau}\in K^{\Gal}\otimes_{\Q_{p}}\Lie(\Gamma_{K})$
(see $\mathsection2$ of \cite{Be16}) where $K^{\Gal}$ is the Galois
closure of $K$. It is defined in such a way that if $W$ is $K^{\Gal}$-linear,
then for $w\in W^{\la}$ and $g\in\Gamma_{K_{n}}$ with $n\gg0$ we
have $g(w)=\sum_{k\in\bbN^{\Sigma_{K}}}l(g)^{k}\frac{\nabla^{k}(w)}{k!}$,
where $l(g)^{k}=\prod_{\tau\in\Sigma_{K}}\tau(l(g))^{k_{\tau}}$ and
$\nabla^{k}(w)=\prod_{\tau\in\Sigma_{K}}\nabla_{\tau}(w)^{k_{\tau}}$.
In other words, we can think of $\tau\circ l$ as giving coordinates
for $\Gamma_{K}$ and $\nabla^{k}(w)$ as being an iterated directional
derivative of $w$. In particular, $W^{K-\la}$ is the subspace of
$W^{\la}$ where $\nabla_{\tau}=0$ for each $\tau\in\Sigma_{K}\backslash\left\{ \Id\right\} $.
On $W^{K-\la}$ (or on $W^{K-\pa}$ if $W$ is Fréchet) we write $\nabla=\nabla_{\Id}$
when there is no danger of confusion; it is given by the forumla
\[
\nabla(w)=\lim_{g\rightarrow1}\frac{g(w)-w}{\log(\chi_{\pi}(g))},
\]
and we have $\nabla(w)=\frac{\log(g)(w)}{\log(\chi_{\pi}(g))}$ when
$g$ is sufficiently close to $1$.

The next lemma is proved in the same way as Proposition 2.2 and Proposition
2.4 of \cite{Be16}.
\begin{lem}
Let $B$ be a Banach (resp. Fréchet) $\Gamma_{K}$-ring and let $W$
be a free $B$-module of finite rank, equipped with a compatible action
of $\Gamma_{K}$. If the $B$ module has a basis $w_{1},...,w_{d}$
in which the function $\Gamma_{K}\rightarrow\GL_{d}(B)\subset M_{d}(B),$
$g\mapsto\Mat(g)$ is locally $K$-analytic (resp. pro $K$-analytic),
then $W^{K-\la}=\oplus_{j=1}^{d}B^{K-\la}w_{i}$ (resp. $W^{K-\pa}=\oplus_{j=1}^{d}B^{K-\pa}w_{i}$).
\end{lem}

\subsection{Locally analytic vectors in $\widehat{K}_{\infty}$-semilinear representations }

Let $L$ be a finite extension of $K$, and write $\Gamma_{L}=\Gal(L_{\infty}/L)$
where $L_{\infty}=LK_{\infty}$. Recall that following result (Proposition
2.10 of \cite{Be16}):
\begin{prop}
$\widehat{L}_{\infty}^{K-\la}=L_{\infty}$. 
\end{prop}

The purpose of this subsection is to prove a similar descent result
for representations.
\begin{thm}
Let $W$ be a finite dimensional $\widehat{L}_{\infty}$-semilinear
representation of $\Gamma_{L}$. Then the natural map $\widehat{L}_{\infty}\otimes_{L_{\infty}}W^{K-\la}\rightarrow W$
is an isomorphism.
\end{thm}

This is proved in $\mathsection4$ of \cite{BC16} under the assumption
that $K$ is Galois over $\Q_{p}$. Here we shall adapt the methods
of ibid. to get rid of this assumption. 

First, we reduce to the case where $L$ is Galois over $\Q_{p}$.
\begin{lem}
Suppose that Theorem 2.3 holds for $M=L^{\Gal}$, the Galois closure
of $L$ over $\Q_{p}$. Then Theorem 2.3 holds for $L$.
\end{lem}

\begin{proof}
Let $L$ be any finite extension of $K$ and let $W$ be a finite
dimensional $\widehat{L_{\infty}}$-semilinear representation of $\Gamma_{K}$.
Write $W_{M}=\widehat{M}_{\infty}\otimes_{\widehat{L}_{\infty}}W$,
so that $W_{M}$ is a finite dimensional $\widehat{M}_{\infty}$-semilinear
representation of $\Gamma_{M}$. Note that $W_{M}$ is actually endowed
with a semilinear $\Gal(M_{\infty}/L)$-action, which restricts to
a $\Gamma_{M}$ action. By the assumption, we have $\widehat{M}_{\infty}\otimes_{M_{\infty}}W_{M}^{K-\la}\cong W_{M}.$
On the other hand, the extension $\Gal(M_{\infty}/L_{\infty})$ is
finite, so we are in the setting for completed Galois descent (see
$\mathsection2.2$ of \cite{BC09}). We have 
\[
W^{K-\la}=W_{M}^{K-\la}\cap W=\left(W_{M}^{K-\la}\right)^{\Gal(M_{\infty}/L_{\infty})}
\]
which implies that $\widehat{M}_{\infty}\otimes_{M_{\infty}}W_{M}^{K-\la}\cong\widehat{M}_{\infty}\otimes_{L_{\infty}}W^{K-\la}$.
We then have the following chain of natural isomorphisms 

\[
\begin{aligned}\widehat{L}_{\infty}\otimes_{L_{\infty}}W^{K-\la} & \cong\left(\widehat{M}_{\infty}\otimes_{L_{\infty}}W^{K-\la}\right)^{\Gal(M_{\infty}/L_{\infty})}\\
 & \cong\left(\widehat{M}_{\infty}\otimes_{M_{\infty}}W_{M}^{K-\la}\right)^{\Gal(M_{\infty}/L_{\infty})}\\
 & \cong\left(W_{M}\right)^{\Gal(M_{\infty}/L_{\infty})}\\
 & \cong W
\end{aligned}
\]
whose composition is the natural map $\widehat{L}_{\infty}\otimes_{L_{\infty}}W^{K-\la}\rightarrow W$,
which proves the claim.
\end{proof}
\begin{prop}
If $\tau\in\Sigma_{K}\backslash\left\{ \mathrm{Id}\right\} $ and
$K^{\Gal}\subset L$, there exists an element $x_{\tau}\in\widehat{L}_{\infty}$
such that $g(x_{\tau})=x_{\tau}+\tau(l(g))$ for $g\in G_{K^{\Gal}}$.
In particular, $\nabla_{\tau}(x_{\tau})=1$ and $\nabla_{\sigma}(x_{\tau})=0$
for $\sigma\neq\tau$.
\end{prop}

\begin{proof}
By $\mathsection3.2$ of \cite{Fo09}, there exists an element $\xi_{\tau}\in\C_{p}^{\times}$
such that $\xi_{\tau}/g(\xi_{\tau})=\tau(\chi_{\pi}(g))$ for $g\in G_{K^{\Gal}}$.
This equation makes it clear that $\xi_{\tau}$ lies in the completion
of $K^{\Gal}K_{\infty}$, which is contained in $\widehat{L}_{\infty}$.
Now take $x_{\tau}=-\log\xi_{\tau}$.
\end{proof}
For each $n\geq1$ and for each $\tau\neq\Id$, choose $x_{\tau}$
as in Proposition 2.5 and let $x_{n,\tau}\in L_{\infty}$ be such
that $||x_{\tau}-x_{n,\tau}||\leq p^{-n}$. For $k\in\bbN^{\Sigma_{K}\backslash\left\{ \Id\right\} }$
we write $(x-x_{n})^{k}=\prod_{\tau\in\Sigma_{K}\backslash\left\{ \Id\right\} }(x_{\tau}-x_{n,\tau})^{k_{\tau}}$. 

\emph{Proof of Theorem 2.3. }By Lemma 2.4, we may assume $L$ is Galois
over $\Q_{p}$. Recall that by Theorem 1.7 of \cite{BC16}, the natural
map $\widehat{L}_{\infty}\otimes_{\widehat{L}_{\infty}^{\la}}W^{\la}\rightarrow W$
is an isomorphism. Therefore, it is enough to prove that the natural
map $\widehat{L}_{\infty}^{\la}\otimes_{L_{\infty}}W^{K-\la}\rightarrow W^{\la}$
is an isomorphism. To prove injectivity, take $\sum_{i=1}^{n}\alpha_{i}\otimes x_{i}\in\widehat{L}_{\infty}^{\la}\otimes_{L_{\infty}}W^{K-\la}$
of minimal length such that $\sum_{i=1}^{n}\alpha_{i}x_{i}=0$. We
may assume that $\alpha_{1}=1$. For each $\tau\neq\Id$, we have
$\nabla_{\tau}\left(\sum_{i=1}^{n}\alpha_{i}x_{i}\right)=\sum_{i=2}^{n}\nabla_{\tau}(\alpha_{i})x_{i}$,
so by minimality $\nabla_{\tau}(\alpha_{i})=0$ for all $i$. This
means that each $\alpha_{i}\in L_{\infty}$, so $\sum_{i=1}^{n}\alpha_{i}\otimes x_{i}=0$.

To prove surjectivity, we give a sketch, omitting all details of convergence;
these can be provided in exactly the same way as in $\mathsection4$
of \cite{BC16}. For each $z\in W^{\la}$, and for each $i\in\bbN^{\Sigma_{K}\backslash\left\{ \Id\right\} }$,
let
\[
y_{i}=\sum_{k\in\bbN^{\Sigma_{K}\backslash\left\{ \Id\right\} }}(-1)^{|k|}(x-x_{n})^{k}\frac{\nabla^{k+i}(z)}{(k+i)!}{k+i \choose k}.
\]
One can show that $y_{i}\in W^{\la}$. By Proposition 2.5, for each
$\tau\in\Sigma_{K}\backslash\left\{ \Id\right\} $ we have $\nabla_{\tau}\left((x-x_{n})^{k}\right)=k_{\tau}(x-x_{n})^{k-1_{\tau}}$,
where $1_{\tau}$ is the tuple $\left(k_{\sigma}\right)\in\bbN^{\Sigma_{K}}$
with $k_{\tau}=1$ and $k_{\sigma}=0$ for $\sigma\neq\tau$. By a
direct calculation this implies that $\nabla_{\tau}(y_{i})=0$, so
that $y_{i}\in W^{K-\la}$. Finally, the identity
\[
z=\sum_{i\in\bbN^{\Sigma_{K}\backslash\left\{ \Id\right\} }}y_{i}(x-x_{n})^{i}
\]
shows that $z\in\widehat{L}_{\infty}\otimes_{\widehat{L}_{\infty}^{\la}}W^{\la}$.$\hfill\ensuremath{\Box}$

\subsection{Pro analytic vectors in $\protect\B_{\protect\dR}$}

The ring $\B_{\dR}^{+}$ contains an element $t_{K}$ for which each
$g\in G_{K}$ acts by $g(t_{K})=\chi_{\pi}(g)t_{K}$. It differs from
the usual $t$ by a unit, but it has the advantage that it carries
an action of $\Gamma_{K}$, which is moreover $K$-analytic. As $\B_{\dR}^{+}/t_{K}\cong\C_{p}$,
the quotients $\left(\B_{\dR}^{+}/t_{K}^{l}\right)^{H_{K}}$ for $l\geq1$
are Banach $\Gamma_{K}$-rings. The ring and $\left(\B_{\dR}^{+}\right)^{H_{K}}$
is a Fréchet $\Gamma_{K}$-ring and $\left(\B_{\dR}\right)^{H_{K}}$
is an LF $\Gamma_{K}$-ring.
\begin{prop}
1.$\left(\B_{\dR}^{+}/t_{K}^{l}\right)^{H_{K},K-\la}=K_{\infty}[t_{K}]/t_{K}^{l}$.

2. $\left(\B_{\dR}^{+}\right)^{H_{K},K-\pa}=K_{\infty}[[t_{K}]]$.

3. $\left(\B_{\dR}\right)^{H_{K},K-\pa}=K_{\infty}\left(\left(t_{K}\right)\right)$.
\end{prop}

\begin{proof}
(3) follows from (2) and (2) follows from (1). To prove (1), we argue
by induction. For $l=1$, this is Proposition 2.2. For $l\geq2$,
we have a short exact sequence
\[
0\rightarrow\C_{p}(l-1)\rightarrow\B_{\dR}^{+}/t_{K}^{l}\rightarrow\B_{\dR}^{+}/t_{K}^{l-1}\rightarrow0.
\]
Taking $H_{K}$ invariants and $K$-locally analytic vectors is left
exact, so we have
\[
0\rightarrow K_{\infty}(l-1)\rightarrow\left(\B_{\dR}^{+}/t_{K}^{l}\right)^{H_{K},K-\la}\rightarrow\left(\B_{\dR}^{+}/t_{K}^{l-1}\right)^{H_{K},K-\la}=K_{\infty}[t_{K}]/t_{K}^{l-1}.
\]
This shows that $\dim_{K_{\infty}}\left(\B_{\dR}^{+}/t_{K}^{l}\right)^{H_{K},K-\la}\leq l$,
so the containment $K_{\infty}[t_{K}]/t_{K}^{l}\subset\left(\B_{\dR}^{+}/t_{K}^{l}\right)^{H_{K},K-\la}$
has to be an equality, concluding the proof.
\end{proof}

\section{Lubin-Tate $p$-adic Hodge theory}

To goal of this section is to provide constructions and properties
of several of Fontaine's functors on $p$-adic representations where
$\Q_{p}$-coefficients are systematically replaced by $K$-coefficients.
Recall from $\mathsection1.2$ that $E$ is a finite extension of
$K$. Throughout, we fix an $E$-linear $G_{K}$-representation $V$
of dimension $d$ over $K$.

\subsection{The modules $\protect\D_{\protect\Sen,K}$ and $\protect\D_{\protect\dif,K}$}

When $K=\Q_{p}$, the modules $\D_{\Sen,K}$ and $\D_{\dif,K}$ can
be defined using the method of Sen (see $\mathsection4$ of \cite{BC08}).
It is unavailable for $K\neq\Q_{p}$, so we make use of locally analytic
and pro analytic vectors instead.

We set $W_{+,l}=\left(\B_{\dR}^{+}/t_{K}^{l}\otimes_{K}V\right)^{H_{K}}$
for $l\geq1$, $W_{+}=\left(\B_{\dR}^{+}\otimes_{K}V\right)^{H_{K}}$
and $W=\left(\B_{\dR}\otimes_{K}V\right)^{H_{K}}$. By Proposition
2.6, we have $K_{\infty}[t_{K}]/t_{K}^{l}$-submodules $\D_{\dif,K}^{+,l}(V)=W_{+,l}^{K-\la}$
for $l\geq1$, a $K_{\infty}[[t_{K}]]$-submodule $\D_{\dif,K}^{+}(V)=W_{+}^{K-\pa}$
and a $K_{\infty}\left(\left(t_{K}\right)\right)$-vector space $\D_{\dif,K}(V)=W^{K-\pa}$.
The subspace $\D_{\dif,K}^{+,1}(V)$ is also called $\D_{\Sen,K}(V)$,
and was already constructed in \cite{BC16}.
\begin{lem}
The natural map $\B_{\dR}^{+}/t_{K}^{l}\otimes_{K_{\infty}[t_{K}]/t_{K}^{l}}W_{+,l}\rightarrow\B_{\dR}^{+}/t_{K}^{l}\otimes_{K}V$
is an isomorphism. 
\end{lem}

\begin{proof}
It suffices to prove that $\H^{1}\left(H_{K},\GL_{d}\left(\B_{\dR}^{+}/t_{K}^{l}\right)\right)=1$.
When $l=1$ this is true by almost étale descent. For $l\geq2$, we
have a short exact sequence
\[
1\rightarrow I+t_{K}^{l-1}\mathrm{M}_{d}\left(\B_{\dR}^{+}/t_{K}^{l}\right)\rightarrow\GL_{d}\left(\B_{\dR}^{+}/t_{K}^{l}\right)\rightarrow\GL_{d}\left(\B_{\dR}^{+}/t_{K}^{l-1}\right)\rightarrow1.
\]
As $I+t_{K}^{l-1}\mathrm{M}_{d}\left(\B_{\dR}^{+}/t_{K}^{l}\right)\cong\mathrm{M}_{d}\left(\C_{p}(l-1)\right)$,
the group $\H^{1}\left(H_{K},I+t_{K}^{l-1}\mathrm{M}_{d}\left(\B_{\dR}^{+}/t_{K}^{l}\right)\right)$
is trivial, and we conclude by induction.
\end{proof}
We will also need the following.
\begin{lem}
Let $w\in W_{+,l}$ and suppose that $t_{K}\cdot w\in\D_{\dif,K}^{+,l}(V)$.
Then $w\in\D_{\dif,K}^{+,l}(V)$.
\end{lem}

\begin{proof}
Since $\nabla(t_{K}w)=t_{K}(w+\nabla(w))$, we have that $\frac{\nabla^{k}(t_{K}w)}{k!}$
is divisible by $t_{K}$ for $k\geq1$. If $t_{K}\cdot w\in\D_{\dif,K}^{+,l}(V)$,
there exists an $n\gg0$ such that for $g\in\Gamma_{n}$, we have
$g(t_{K}w)=\sum_{k\geq0}l(g)^{k}t_{K}w_{k}$, where $w_{k}=t_{K}^{-1}\frac{\nabla^{k}(t_{K}w)}{k!}$.
Therefore, $g(w)=\chi_{\pi}(g^{-1})\sum_{k\geq0}l(g)^{k}w_{k}$, so
$w$ is locally $K$-analytic. 
\end{proof}
\begin{prop}
1. The natural map $\B_{\dR}^{+}/t_{K}^{l}\otimes_{K_{\infty}[t_{K}]/t_{K}^{l}}\D_{\dif,K}^{+,l}(V)\rightarrow\B_{\dR}^{+}/t_{K}^{l}\otimes_{K}V$
is an isomorphism, and $\D_{\dif,K}^{+,l}(V)$ is a free $K_{\infty}[t_{K}]/t_{K}^{l}$-module
of rank $d$.

2. The natural map $\B_{\dR}^{+}\otimes_{K_{\infty}[[t_{K}]]}\D_{\dif,K}^{+}(V)\rightarrow\B_{\dR}^{+}\otimes_{K}V$
is an isomorphism, and $\D_{\dif,K}^{+}(V)$ is a free $K_{\infty}[[t_{K}]]$-module
of rank $d$.

3. The natural map $\B_{\dR}\otimes_{K_{\infty}\left(\left(t_{K}\right)\right)}\D_{\dif,K}(V)\rightarrow\B_{\dR}\otimes_{K}V$
is an isomorphism, and $\D_{\dif,K}(V)$ is a $K_{\infty}\left(\left(t_{K}\right)\right)$-vector
space of dimension $d$.
\end{prop}

\begin{proof}
Recall that $\D_{\dif,K}^{+,l}(V)=W_{+,l}^{K-\la}$. By Lemma 3.1,
proving (1) reduces to showing that the natural map $\widehat{K}_{\infty}\otimes_{K_{\infty}}W_{+,l}^{K-\la}\rightarrow W_{+,l}$
is an isomorphism and that $W_{+,l}^{K-\la}$ is a free $K_{\infty}[t_{K}]/t_{K}^{l}$-module
of rank $d$. By Theorem 2.3, this is true if $l=1$. For $l\geq2$,
we have a short exact sequence
\[
0\rightarrow W_{+,1}^{K-\la}(l-1)\rightarrow W_{+,l}^{K-\la}\rightarrow W_{+,l-1}^{K-\la}.
\]
By the case $l=1$, we know that $W_{+,1}^{K-\la}(l-1)$ contains
linearly independent elements $e_{1},...,e_{d}$ which are all divisible
by $t_{K}^{l-1}$. Writing $f_{i}=t_{K}^{1-l}e_{i}$ for $1\leq i\leq d$,
the elements $f_{1},...,f_{d}$ span a free submodule $W_{+,l}^{'}$
of $W_{+,l}$ which surjects onto $W_{+,l-1}$ and which contains
$W_{+,1}$; so $W_{+,l}^{'}=W_{+,l}$. It now suffices to show that
the $f_{i}$ are locally $K$-analytic, and this follows from Lemma
3.2. This concludes the proof of (1).

As each $W_{+,l}^{K-\la}$ is a free $K_{\infty}[t_{K}]/t_{K}^{l}$-module
of rank $d$, we have that $W_{+}^{K-\pa}=\underset{\leftarrow}{\lim}W_{+,l}^{K-\la}$
is a free $K_{\infty}[[t_{K}]]$-module of rank $d$, and the chain
of isomorphisms

\[
\begin{aligned}\B_{\dR}^{+}\otimes_{K_{\infty}[[t_{K}]]}\D_{\dif,K}^{+}(V) & \cong\B_{\dR}^{+}\otimes_{K_{\infty}[[t_{K}]]}W^{K-\pa}\\
 & \cong\underset{\leftarrow}{\lim}\left(\B_{\dR}^{+}/t_{K}^{l}\otimes_{K_{\infty}[t_{K}]/t_{K}^{l}}W_{l}^{K-\la}\right)\\
 & \cong\underset{\leftarrow}{\lim}\left(\B_{\dR}^{+}/t_{K}^{l}\otimes_{K}V\right)\\
 & \cong\B_{\dR}^{+}\otimes_{K}V,
\end{aligned}
\]
whose composition is the natural map $\B_{\dR}^{+}\otimes_{K_{\infty}[[t_{K}]]}\D_{\dif,K}^{+}(V)\rightarrow\B_{\dR}^{+}\otimes_{K}V$.
This proves (2), and (3) follows immediately since $\D_{\dif,K}(V)=\mathrm{colim}_{i}\D_{\dif,K}^{+}(V\left(i\right))$.
\end{proof}
Recall from $\mathsection2$ that the modules $\D_{\Sen,K}$ and $\D_{\dif,K}$
are both endowed with a canonical differential operator. We write
$\Theta_{\Sen,K}$,$\nabla_{\dif,K}$ respectively for the operators
acting on $\D_{\Sen,K}$,$\D_{\dif,K}$. The operator $\Theta_{\Sen,K}$
is $K_{\infty}$-linear, while $\nabla_{\dif,K}$ is a derivation
over $\nabla_{K_{\infty}\left(\left(t_{K}\right)\right)}=t_{K}\frac{\partial}{\partial t_{K}}$.

The following result serves to complete the analogy with the usual
$\D_{\mathrm{Sen}}$.
\begin{prop}
The following are equivalent.

1. $\C_{p}\otimes_{K}V\cong\oplus_{i=1}^{d}\C_{p}(\chi_{\pi}^{n_{i}})$,
where the $n_{i}\in\Z$.

2. $\D_{\Sen,K}(V)\cong\oplus_{i=1}^{d}K_{\infty}(\chi_{\pi}^{n_{i}})$,
where the $n_{i}\in\Z$.

3. $\Theta_{\Sen,K}$ is semisimple with integer eigenvalues $\left\{ n_{i}\right\} _{i=1}^{d}$.
\end{prop}

\begin{proof}
Suppose $v_{1},...,v_{d}$ is a basis of $\C_{p}\otimes_{K}V$ for
which $g(v_{i})=\chi_{\pi}^{n_{i}}(g)v_{i}$ for $g\in G_{K}$. The
action of $G_{K}$ on each $v_{i}$ factors through $\Gamma_{K}$
and is locally $K$-analytic, so $v_{i}\in\D_{\Sen,K}(V)$, which
shows (1) implies (2). Next, (2) implies (3) because the action of
$\Theta_{\Sen,K}$ on $K_{\infty}(\chi_{\pi}^{n_{i}})$ is given by
multiplication with $n_{i}$. Finally, suppose that (3) holds, and
let $v_{1},...,v_{d}$ be a basis of $\D_{\Sen,K}(V)$ for which $\Theta_{\Sen,K}(v_{i})=n_{i}v_{i}$.
By integrating the action of $\Gamma_{K}$, we see that $g\in\Gamma_{K}$
acts by $g(v_{i})=\eta_{i}(g)\chi_{\pi}^{n_{i}}(g)v_{i}$, where $\eta_{i}$
is a finite order character of $\Gamma_{K}$. Then $\C_{p}(\eta_{i}\chi_{\pi}^{n_{i}})\cong\C_{p}(\chi_{\pi}^{n_{i}})$,
so

\[
\begin{aligned}\C_{p}\otimes_{K}V & \cong\C_{p}\otimes_{K_{\infty}}\D_{\Sen,K}(V)\\
 & \cong\oplus_{i=1}^{d}\C_{p}(\eta_{i}\chi_{\pi}^{n_{i}})\\
 & \cong\oplus_{i=1}^{d}\C_{p}(\chi_{\pi}^{n_{i}}).
\end{aligned}
\]
\end{proof}
If the conditions of Proposition 3.4 hold for some $n_{i}\in\Z$,
the $n_{i}$ are called the $K$-Hodge-Tate weights of $V$.

\subsection{The modules $\protect\D_{\protect\HT,K}$ and $\protect\D_{\protect\dR,K}$}

Let $\B_{\HT,K},\B_{\dR,K},\B_{\dR,K}$ respectively be the rings
$\C_{p}[t_{K},t_{K}^{-1}],\B_{\dR}^{+}$ and $\B_{\dR}$. We set $\D_{\HT,K}(V)=\left(\B_{\HT,K}\otimes_{K}V\right)^{G_{K}}$,
$\D_{\dR,K}^{+}(V)=\left(\B_{\dR,K}^{+}\otimes_{K}V\right)^{G_{K}}$
and $\D_{\dR,K}(V)=\left(\B_{\dR,K}\otimes_{K}V\right)^{G_{K}}$.
We say that $V$ is $K$-Hodge-Tate (resp. positive $K$-de Rham,
resp. $K$-de Rham) if $\dim_{K}\D_{\HT,K}(V)=d$ (resp. $\dim_{K}\D_{\dR,K}^{+}(V)=d$,
resp. $\dim_{K}\D_{\dR,K}(V)=d$).
\begin{lem}
$(\C_{p}\otimes_{K}V)^{G_{K}}=\left(\D_{\Sen,K}(V)\right)^{\Gamma_{K}}$.
\end{lem}

\begin{proof}
By Proposition 3.3, we have 
\[
\left(\C_{p}\otimes_{K}V\right)^{G_{K}}=\left(\C_{p}\otimes\D_{\Sen,K}(V)\right)^{G_{K}}=\left(\hat{K}_{\infty}\otimes_{K_{\infty}}\D_{\Sen,K}(V)\right)^{\Gamma_{K}}.
\]
As $\left(\hat{K}_{\infty}\otimes_{K_{\infty}}\D_{\Sen,K}(V)\right)^{\Gamma_{K}}$
is fixed by the action of $\Gamma_{K}$, it is also locally $K$-analytic
on $\Gamma_{K}$, so it is contained in $\left(\hat{K}_{\infty}\otimes_{K_{\infty}}\D_{\Sen,K}(V)\right)^{K-\la}$.
But according to Lemma 2.1, 
\[
\left(\hat{K}_{\infty}\otimes_{K_{\infty}}\D_{\Sen,K}(V)\right)^{K-\la}=\D_{\Sen,K}(V).
\]
\end{proof}
\begin{prop}
1. $\D_{\HT,K}(V)=\oplus_{l\in\Z}\left(\D_{\Sen,K}(V)t_{K}^{l}\right)^{\Gamma_{K}}$.

2. The natural map $\B_{\HT,K}\otimes_{K}\D_{\HT,K}(V)\rightarrow\B_{\HT,K}\otimes_{K}V$
is an isomorphism if $V$ is $K$-Hodge-Tate.

3. $V$ is $\C_{p}$-admissible if and only if $\Theta_{\Sen,K}=0$
on $\D_{\Sen,K}(V)$.
\end{prop}

\begin{proof}
We have $\D_{\HT,K}(V)=\oplus_{l\in\Z}\left(\C_{p}t_{K}^{l}\otimes_{K}V\right)^{G_{K}}$,
so (upon twisting $V$ with an appropriate power of $\chi_{\pi}$)
the equality $\D_{\HT,K}(V)=\oplus_{l\in\Z}\left(\D_{\Sen,K}(V)t_{K}^{l}\right)^{\Gamma_{K}}$
reduces to the verification that $(\C_{p}\otimes_{K}V)^{G_{K}}=\left(\D_{\Sen,K}(V)\right)^{\Gamma_{K}}$,
which was done in the previous lemma. This proves (1). To prove (2),
suppose $V$ is $K$-Hodge-Tate. Then by (1) and Proposition 3.4 we
have $\D_{\Sen,K}(V)\cong\oplus_{i=1}^{d}K_{\infty}(\chi_{\pi}^{n_{i}})$,
which gives the second isomorphism in

\[
\begin{aligned}\B_{\HT,K}\otimes_{K}\D_{\HT,K}(V) & \cong\B_{\HT,K}\otimes_{K}\left(\oplus_{l\in\Z}\left(\D_{\Sen,K}(V)t_{K}^{l}\right)^{\Gamma_{K}}\right)\\
 & \cong\B_{\HT,K}\otimes_{K_{\infty}}\D_{\Sen,K}(V)\\
 & \cong\B_{\HT,K}\otimes_{K}V.
\end{aligned}
\]

Finally, (3) follows from Proposition 3.4.
\end{proof}
By the same logic one obtains similar results for $\D_{\dR,K}$.
\begin{prop}
1. $\D_{\dR,K}(V)=\D_{\dif,K}(V)^{\Gamma_{K}}$.

2. The natural map $\B_{\dR,K}\otimes_{K}\D_{\dR,K}(V)\rightarrow\B_{\dR,K}\otimes_{K}V$
is an isomorphism if $V$ is $K$-de Rham.

3. $V$ is $K$-de Rham if and only if $\nabla_{\dif,K}$ has a full
set of sections on $\D_{\dif,K}(V)$.
\end{prop}

\subsection{The modules $\protect\D_{\protect\cris,K}$ and $\protect\D_{\protect\st,K}$ }

Recall that $\B_{\max}^{+}$ is a period ring similar to Fontaine's
$\B_{\cris}^{+}$. The element\footnote{It is $\log_{\mathcal{F}}u$ for the element $u$ introduced in $\mathsection4.1$
below.} $t_{K}$ introduced in $\mathsection2.3$ actually lies in $\B_{\max}^{+}\otimes_{K_{0}}K$.
Denote by $\B_{\max,K}^{+}$, $\B_{\st,K}^{+}$, $\B_{\max,K}$, $\B_{\st,K}$
respectively the rings $\B_{\max}^{+}\otimes_{K_{0}}K$, $\B_{\st}^{+}\otimes_{K_{0}}K$,
$\B_{\max,K}^{+}\left[\frac{1}{t_{K}}\right]$ and $\B_{\st,K}^{+}\left[\frac{1}{t_{K}}\right]$.
These rings carry a $\varphi_{q}=\varphi^{f}$-action, and the usual
monodromy operator $N$ of $\B_{\st}^{+}$ extends to $\B_{\st,K}$
with $\B_{\st,K}^{N=0}=\B_{\max,K}$. If $L$ is a finite extension
of $K$, we set $\D_{\cris,K}^{+}(V)$, $\D_{\st,K}^{+}(V)$, $\D_{\cris,K}(V)$,
$\D_{\st,K}(V)$ respectively to be $\D_{\cris,K}^{+}(V)=\left(\B_{\max,K}^{+}\otimes_{K}V\right)^{G_{L}}$,
$\D_{\st,K}^{+}(V)=\left(\B_{\st,K}^{+}\otimes_{K}V\right)^{G_{L}}$,
$\D_{\cris,K}(V)=\left(\B_{\max,K}\otimes_{K}V\right)^{G_{L}}$ and
$\D_{\st,K}(V)=\left(\B_{\st,K}\otimes_{K}V\right)^{G_{L}}$. 
\begin{lem}
If $V$ is a $K$-analytic representation, then $\D_{\cris,K}(V)=\left(\B_{\max}\otimes_{K_{0}}V\right)^{G_{L}}$.
\end{lem}

\begin{proof}
By Lemma 8.17 of \cite{Co02}, the element $t_{K}$ divides the usual
$t$ in $\B_{\max,K}^{+}$. Therefore, $\B_{\max,K}=\B_{\max,K}^{+}\left[\frac{1}{t_{K}}\right]$
is contained in $\B_{\max}\otimes_{K_{0}}K=\B_{\max,K}^{+}\left[\frac{1}{t}\right]$,
and it follows that $\D_{\cris,K}(V)\subset\left(\B_{\max}\otimes_{K_{0}}V\right)^{G_{L}}$
for any representation $V$. It suffices to show when $V$ is $K$-analytic,
this inclusion is an equality. If all the integer-valued Hodge-Tate
weights (in the usual sense) of $V$are nonnegative, this is clear,
because

\[
\D_{\cris,K}\left(V\right)=\D_{\cris,K}^{+}\left(V\right)=\left(\B_{\max}^{+}\otimes_{K_{0}}V\right)^{G_{L}}=\left(\B_{\max}\otimes_{K_{0}}V\right)^{G_{L}}.
\]
In general, it suffices to prove $\D_{\cris,K}(V)\subset\left(\B_{\max}\otimes_{K_{0}}V\right)^{G_{L}}$
after twisting $V$ by a power of $\chi_{\pi}$. Given that $V$ is
$K$-analytic, there exists an $n\gg0$ so that all the integer-valued
Hodge-Tate weights of $V\left(\chi_{\pi}^{n}\right)$ are nonnegative,
and this case was already dealt with.
\end{proof}
To a filtered $(\varphi_{q},N)$-module $\D$ over $L_{0}\otimes_{K_{0}}E$
one can associate two polygons. The Hodge polygon $P_{H}(\D)$, whose
slopes have lengths according to the jumps in the filtration; and
the Newton polygon $P_{N}(\D)$, whose slopes match the slopes of
$\varphi_{q}$ with respect to the valuation $\val_{\pi}$. We say
$\D$ is admissible if the endpoints of $P_{H}(\D)$ and $P_{N}(D)$
are the same and if $P_{H}(\D_{0})$ lies below $P_{N}(\D_{0})$ for
every subobject $\D_{0}$ of $\D$.

If $\D$ is a filtered $(\varphi_{q},N)$-module over $L_{0}\otimes_{K_{0}}E$,
then 
\[
\mathrm{I}_{\Q_{p}}^{K}(\D):=\left(L_{0}\otimes_{\Q_{p}}E\right)\left[\varphi\right]\otimes_{\left(L_{0}\otimes_{K_{0}}E\right)\left[\varphi_{q}\right]}\D
\]
is a filtered $\left(\varphi,N\right)$-module over $L_{0}\otimes_{\Q_{p}}E$.
The following is proved in $\mathsection3$ of \cite{KR09} under
the assumption that $N=0$, but the proof in the general case is the
same. Note that in loc. cit. this statement is actually proved for
$\left(\B_{\max}\otimes_{K_{0}}V\right)^{G_{L}}$ instead of $\D_{\cris,K}(V)$,
but these coincide for $K$-analytic representations by Lemma 3.8.
\begin{prop}
Suppose that $V\in\mathrm{Rep}_{E}(G_{L})$ is $K$-analytic. Then
\[
\mathrm{I}_{\Q_{p}}^{K}(\D_{\st,K}(V))=\D_{\st,\Q_{p}}(V).
\]
Furthermore, $\D_{\st,K}(V)$ is admissible if and only if $\D_{\st,\Q_{p}}(V)$
is admissible.
\end{prop}

We say that $V$ is $K$-potentially semistable if for some finite
extension $L$ of $K$ we have $\rank_{L_{0}\otimes_{K_{0}}E}\D_{\st,K}(V)=\dim_{E}V$.

\begin{cor}
Suppose $V$ is a $K$-analytic representation. Then the following
are equivalent.

1. $V$ is de Rham.

2. $V$ is $K$-de Rham.

3. $V$ is potentially semistable.

4. $V$ is $K$-potentially semistable.
\end{cor}

\begin{proof}
The equivalence between (1) and (3) is the $p$-adic monodromy theorem,
which is Theorem 0.7 of \cite{Be02}. The equivalence between (3)
and (4) follows from Proposition 3.9. It remains to prove that (1)
and (2) are equivalent. Indeed, we have $\D_{\dR}(V)\cong\oplus_{\tau\in\Sigma_{K}}\D_{\dR,K}(V^{\tau})$,
with $V^{\tau}$ being the $\tau$-twist of $V$. Since $V$ is assumed
$K$-analytic, the representation $V^{\tau}$ is $\C_{p}$-admissible
for $\tau\neq\Id$, so it also $K$-de Rham. This implies that $\dim_{K}\D_{\dR,K}(V^{\tau})=d$,
so $\dim_{K}\D_{\dR}(V)=\dim_{\Q_{p}}V$ if and only if $\dim_{K}\D_{\dR,K}(V)=\dim_{K}V$,
as required.
\end{proof}
We conclude with a lemma that will be used in the proof of Theorem
6.8. 
\begin{lem}
Suppose that $V\in\mathrm{Rep}_{E}(G_{K})$ is $K$-analytic and that
$L=K$, and let $\alpha\in E^{\times}$. Then $\D_{\cris,\Q_{p}}(V)^{\varphi_{q}=\alpha}=\left(K_{0}\otimes_{\Q_{p}}E\right)\left[\varphi\right]\otimes_{E\left[\varphi_{q}\right]}\D_{\cris,K}(V){}^{\varphi_{q}=\alpha}.$
\end{lem}

\begin{proof}
In general, if $\D$ is a filtered $\varphi_{q}$-module over $E$,
then $\mathrm{I}_{\Q_{p}}^{K}(\D)^{\varphi_{q}=\alpha}=\left(K_{0}\otimes_{\Q_{p}}E\right)\left[\varphi\right]\otimes_{E\left[\varphi_{q}\right]}\D{}^{\varphi_{q}=\alpha}$
because the $\varphi_{q}$ action is $\left(K_{0}\otimes_{\Q_{p}}E\right)\left[\varphi\right]$-linear.
Explicitly, there is a decomposition of $E\left[\varphi_{q}\right]$-modules
\[
\mathrm{I}_{\Q_{p}}^{K}(\D)=\left(K_{0}\otimes_{\Q_{p}}E\right)\left[\varphi\right]\otimes_{E\left[\varphi_{q}\right]}\D=\oplus_{i=0}^{f-1}\left(K_{0}\otimes_{\Q_{p}}E\right)\varphi^{i}\otimes\D,
\]
and taking $\varphi_{q}=\alpha$ of both sides yields the desired
equality. The lemma now follows from Proposition 3.9.
\end{proof}

\section{Big period rings}

In this section, we recall the big period rings which will be used
in the proofs of $\mathsection5$. The most important result of this
section is the determination of pro-$K$-analytic vectors in Theorem
4.6.

If $\mathcal{F}$ is the formal $\mathcal{O}_{K}$-module associated
to $\pi$ as in $\mathsection1.2$, we choose a coordinate $T$ for
$\mathcal{F}$ so that for $a\in\mathcal{O}_{K}$ we have a power
series $[a]=[a](T)$ corresponding to the action of $a$ on $\mathcal{F}$.
For $n\geq0$ we choose elements $u_{n}\in\mathcal{O}_{\C_{p}}$ such
that $u_{0}=0$, $u_{1}\neq0$ and $[\pi](u_{n})=u_{n-1}$.

\subsection{The rings $\widetilde{\protect\B}_{\protect\rig}^{\dagger}$ and
$\widetilde{\protect\B}_{\log}^{\dagger}$}

This subsection provides ramified counterparts for the constructions
given in $\mathsection2$ of \cite{Be02} in the case $K=K_{0}$.
Some of the content of this subsection can also be found in $\mathsection3$
of \cite{Be16}, though our notation is slightly different in parts.
Recall the notations from $\mathsection1.2$ and set $\mathcal{O}_{\C_{p}^{\flat}}=\lim\left(\mathcal{O}_{\C_{p}}/\pi\xleftarrow{x\mapsto x^{q}}\mathcal{O}_{\C_{p}}/\pi\xleftarrow{x\mapsto x^{q}}...\right)$.
We equip $\mathcal{O}_{\C_{p}^{\flat}}$ with the valuation $\left|\left(\overline{x}_{n}\right)_{n\geq0}\right|=\underset{n\rightarrow\infty}{\lim}\left|x_{n}\right|^{q^{n}}$
where $x_{n}\in\mathcal{O}_{\C_{p}}$ is a lift of $\overline{x}_{n}$.
Denote by $\widetilde{\A}_{0}^{+}$, $\widetilde{\A}^{+}$ respectively
the rings $W\left(\mathcal{O}_{\C_{p}^{\flat}}\right)$ and $\widetilde{\A}_{0}^{+}\otimes_{\mathcal{O}_{K_{0}}}\mathcal{O}_{K}$.
Then $\overline{u}=\left(\overline{u}_{n}\right)_{n\geq0}$ lies in
$\in\mathcal{O}_{\C_{p}^{\flat}}$, and by $\mathsection8$ of \cite{Co02}
there exists an element $u\in\widetilde{\A}^{+}$ which lifts $\overline{u}$
and which satisfies $\varphi_{q}(u)=[\pi](u)$ and $g(u)=[\chi_{\pi}(g)](u)$
for $g\in\Gamma_{K}$.

Let $\varpi\in\mathcal{O}_{\C_{p}^{\flat}}$ be any element with $|\varpi|=p^{-p/p-1}$.
Given $r,s\in\Z_{\geq0}[1/p]$ with $r\leq s$, and given $\mathrm{A}\in\left\{ \widetilde{\A}_{0},\widetilde{\A}\right\} $,
we set
\[
\mathrm{A}^{\left[r,s\right]}=\mathrm{A}^{+}\left\langle \frac{p}{\left[\varpi\right]^{r}},\frac{\left[\varpi\right]^{s}}{p}\right\rangle ,
\]
the completion of $\mathrm{A}^{+}\left[\frac{p}{\left[\varpi\right]^{r}},\frac{\left[\varpi\right]^{s}}{p}\right]$
with respect to the $\left(p,\left[\varpi\right]\right)$-adic topology.\footnote{In some references the completion is taken with respect to the $p$-adic
topology, but this makes no difference because $p$ divides a power
of $\left[\varpi\right]$.} We write $\widetilde{\B}_{0}^{[r,s]}=\widetilde{\A}_{0}^{[r,s]}[1/p]$
and $\widetilde{\B}^{[r,s]}=\widetilde{\A}^{[r,s]}[1/\pi]$.
\begin{lem}
1. $\widetilde{\A}_{0}^{[r,s]}\otimes_{\mathcal{O}_{K_{0}}}\mathcal{O}_{K}=\widetilde{\A}^{[r,s]}$.

2. $\widetilde{\B}_{0}^{[r,s]}\otimes_{K_{0}}K=\widetilde{\B}^{[r,s]}.$
\end{lem}

\begin{proof}
(2) follows from (1). To prove (1), we write
\[
\widetilde{\A}^{+}\left[\frac{p}{\left[\varpi\right]^{r}},\frac{\left[\varpi\right]^{s}}{p}\right]=\oplus_{i=0}^{e-1}\pi^{i}\widetilde{\A}_{0}^{+}\left[\frac{p}{\left[\varpi\right]^{r}},\frac{\left[\varpi\right]^{s}}{p}\right]
\]
as $\widetilde{\A}_{0}^{+}$-modules. Now take the $\left(p,\left[\varpi\right]\right)$-adic
completion of both sides.
\end{proof}
We denote by $\widetilde{\B}_{\rig,0}^{\dagger,r}$,$\widetilde{\B}_{\rig}^{\dagger,r}$,
$\widetilde{\B}_{\rig,0}^{\dagger}$ and $\widetilde{\B}_{\rig}^{\dagger}$
respectively the rings $\cap_{r\leq s}\widetilde{\B}_{0}^{[r,s]},\cap_{r\leq s}\widetilde{\B}^{[r,s]},\cup_{r>0}\widetilde{\B}_{\rig,0}^{\dagger,r}$
and $\cup_{r>0}\widetilde{\B}_{\rig}^{\dagger,r}$. The $\varphi$
and $G_{K}$ actions on $\widetilde{\A}_{0}^{+}$ (resp. the $\varphi_{q}$
and $G_{K}$ actions on $\widetilde{\A}^{+}$) extend to $\widetilde{\B}_{\rig,0}^{\dagger}$
(resp. to $\widetilde{\B}_{\rig}^{\dagger}$). The following is proved
in Proposition 2.23 of \cite{Be02} in the case $K=\Q_{p}$, but the
same proof works in the general case.
\begin{prop}
There exists a unique map $\log:\widetilde{\A}^{+}\rightarrow\widetilde{\B}_{\rig}^{\dagger}$$\left[X\right]$
satisfying $\log(\pi)=0$, $\log\left[\overline{u}\right]=X$, $\log\left[x\right]=0$
for $x\in\overline{\mathbb{F}}_{q}$ and $\log(xy)=\log(x)+\log(y)$,
such that if $\left[x\right]-1$ is sufficiently close to $1$, we
have
\[
\log\left[x\right]=\sum_{n\geq1}\left(-1\right)^{n-1}\frac{\left(\left[x\right]-1\right)^{n}}{n}.
\]

Moreover, if $x\in\mathcal{O}_{\C_{p}^{\flat}}^{\times}$ then $\log\left[x\right]\in\widetilde{\B}_{\rig}^{\dagger}$. 
\end{prop}

Write $p^{\flat}$ and $\pi^{\flat}$ for the elements $\left(\overline{p},\overline{p^{1/q}},...\right)$
and $\left(\overline{\pi},\overline{\pi^{1/q}},...\right)$ of $\mathcal{O}_{\C_{p}^{\flat}}$.
We set $\widetilde{\B}_{\log,0}^{\dagger}=\widetilde{\B}_{\rig,0}^{\dagger}\left[\log\left[p^{\flat}\right]\right]$
and $\widetilde{\B}_{\log}^{\dagger}=\widetilde{\B}_{\rig}^{\dagger}\left[\log\left[\pi^{\flat}\right]\right]$.
Since $u/\left[\overline{u}\right]\equiv1\mod\pi$, we have $\log\left(u/\left[\overline{u}\right]\right)\in\widetilde{\B}_{\rig}^{\dagger}$;
on the other hand, $\overline{u}^{q-1}/\left(\pi^{\flat}\right)^{q}$
is a unit of $\mathcal{O}_{\C_{p}^{\flat}}^{\times}$, so $\log\left[\overline{u}\right]^{q-1}/\left[\pi^{\flat}\right]^{q}\in\widetilde{\B}_{\rig}^{\dagger}$
as well. Combining these two observations, we see that in $\widetilde{\B}_{\log}^{\dagger}$
we have
\[
q\log\left[\pi^{\flat}\right]=(q-1)\log u-(q-1)\log\left(u/\left[\overline{u}\right]\right)-\log\left[\overline{u}\right]^{q-1}/\left[\pi^{\flat}\right]^{q}
\]
\[
\equiv(q-1)\log u\mod\widetilde{\B}_{\rig}^{\dagger},
\]
 so we also have $\widetilde{\B}_{\log}^{\dagger}=\widetilde{\B}_{\rig}^{\dagger}\left[\log u\right]$.

The $\varphi$ (resp. $\varphi_{q}$) action on $\widetilde{\B}_{\rig,0}^{\dagger}$
(resp. on $\widetilde{\B}_{\rig}^{\dagger}$) extends to $\widetilde{\B}_{\log,0}^{\dagger}$
(resp. to $\widetilde{\B}_{\log}^{\dagger}$) by setting $\varphi(\log\left[p^{\flat}\right])=p\log\left[p^{\flat}\right]$
and $g(\log\left[p^{\flat}\right])=\log\left[g\left(p^{\flat}\right)\right]$
(resp. $\varphi_{q}(\log\left[\pi^{\flat}\right])=q\log\left[\pi^{\flat}\right]$
and $g(\log\left[\pi^{\flat}\right])=\log\left[g\left(\pi^{\flat}\right)\right]$).
We have a monodromy operator $N$ which acts on $\widetilde{\B}_{\log,0}^{\dagger}$
(resp. on $\widetilde{\B}_{\log}^{\dagger}$) by $-\frac{d}{d\log\left[p^{\flat}\right]}$
(resp. by $-\frac{1}{e}\frac{d}{d\log\left[\pi^{\flat}\right]}$),
and $N\varphi=p\varphi N$ (resp. $N\varphi_{q}=q\varphi_{q}N$).
\begin{prop}
1. $\widetilde{\B}_{\rig,0}^{\dagger}\otimes_{K_{0}}K=\widetilde{\B}_{\rig}^{\dagger}$.

2. $\widetilde{\B}_{\log,0}^{\dagger}\otimes_{K_{0}}K=\widetilde{\B}_{\log}^{\dagger}$.
\end{prop}

\begin{proof}
For $r\leq s$ we have by Lemma 4.1 that $\widetilde{\B}_{0}^{[r,s]}\otimes_{K_{0}}K=\widetilde{\B}^{[r,s]}$.
As $K$ is finite free over $K_{0}$, this implies 

\[
\begin{aligned}\widetilde{\B}_{\rig,0}^{\dagger}\otimes_{K_{0}}K & =\left(\cap_{r\leq s}\widetilde{\B}_{0}^{[r,s]}\right)\otimes_{K_{0}}K\\
 & =\cap_{r\leq s}\left(\widetilde{\B}_{0}^{[r,s]}\otimes_{K_{0}}K\right)\\
 & =\cap_{r\leq s}\widetilde{\B}^{[r,s]}\\
 & =\widetilde{\B}_{\rig}^{\dagger}.
\end{aligned}
\]
For (2), we write $\pi^{e}=pv$ with $v\in\mathcal{O}_{K}^{\times}$.
We can find $v^{\flat}=\left(\overline{v},\overline{v^{1/q}},...\right)\in\mathcal{O}_{\C_{p}^{\flat}}$
such that $\left[\pi^{\flat}\right]^{e}=\left[p^{\flat}\right]\left[v^{\flat}\right]$,
so $e\log\left[\pi^{\flat}\right]\equiv\log\left[p^{\flat}\right]\mod\widetilde{\B}_{\rig}^{\dagger}$,
and

\[
\begin{aligned}\widetilde{\B}_{\log}^{\dagger} & =\widetilde{\B}_{\rig}^{\dagger}\left[\log\left[p^{\flat}\right]\right]\\
 & =\left(\widetilde{\B}_{\rig,0}^{\dagger}\otimes_{K_{0}}K\right)\left[\log\left[p^{\flat}\right]\right]\\
 & =\widetilde{\B}_{\log,0}^{\dagger}\otimes_{K_{0}}K.
\end{aligned}
\]
.
\end{proof}

\subsection{Pro $K$-analytic vectors}

Let $\B_{\rig,K}^{\dagger}$ be the Robba ring, i.e. the ring of power
series $f(T)=\sum_{k\in\Z}a_{k}T^{k}$ with $a_{k}\in K$ and such
that $f(T)$ converges on some nonempty annulus $r<|T|<1$. The ring
$\B_{\rig,K}^{\dagger}$ can be viewed as a subring of $\widetilde{\B}_{\rig,K}^{\dagger}=\left(\widetilde{\B}_{\rig}^{\dagger}\right)^{H_{K}}$
by identifying $T$ with the element $u$ of $\mathsection4.1$. It
has induced $\varphi_{q}$ and $\Gamma_{K}$ actions.

Recall the following result (Theorem B of \cite{Be16}), which determines
the ring of pro $K$-analytic vectors in $\widetilde{\B}_{\rig,K}^{\dagger}$.
\begin{thm}
$\left(\widetilde{\B}_{\rig,K}^{\dagger}\right)^{K-\pa}=\cup_{n\geq0}\varphi_{q}^{-n}\left(\B_{\rig,K}^{\dagger}\right).$
\end{thm}

On the other hand, we can also write $\widetilde{\B}_{\log,K}^{\dagger}=\left(\widetilde{\B}_{\log}^{\dagger}\right)^{H_{K}}$.
The goal of this subsection is to obtain an analogous result for $\widetilde{\B}_{\log,K}^{\dagger}$.
\begin{prop}
We have $\log u\in\left(\widetilde{\B}_{\log,K}^{\dagger}\right)^{K-\pa}$.
\end{prop}

Before we give a proof of Proposition 4.5, we record the following
consequence. Let $\B_{\log,K}^{\dagger}=\B_{\rig,K}^{\dagger}\left[\log T\right]$,
thought of as a subring of $\widetilde{\B}_{\log,K}^{\dagger}$. The
$\varphi_{q}$ action on $\log T$ is given by $\varphi_{q}(\log T)=q\log T+\log\left(\left[\pi\right](T)/T^{q}\right)$,
where $\log\left(\left[\pi\right](T)/T^{q}\right)\in\B_{\rig,K}^{\dagger}$.
\begin{thm}
$\left(\widetilde{\B}_{\log,K}^{\dagger}\right)^{K-\pa}=\cup_{n\geq0}\varphi_{q}^{-n}\left(\B_{\log,K}^{\dagger}\right)$.
\end{thm}

\begin{proof}
Fix $d\geq0$. As $g\left(\log u\right)=\log u+\log\frac{g(u)}{u}$
for $g\in\Gamma_{K}$, the submodule $\oplus_{i=0}^{d}\widetilde{\B}_{\rig,K}^{\dagger}\cdot\left(\log u\right)^{i}$
is closed under the $\Gamma_{K}$-action. By Proposition 4.5, the
elements $1,\log u,...,\left(\log u\right)^{i}$ from a $\widetilde{\B}_{\rig,K}^{\dagger}$-basis
of this submodule for which the action is pro $K$-analytic. Combining
Lemma 2.1 and Theorem 4.4, we obtain

\[
\begin{aligned}\left(\oplus_{i=0}^{d}\widetilde{\B}_{\rig,K}^{\dagger}\cdot\left(\log u\right)^{i}\right)^{K-\pa} & =\oplus_{i=0}^{d}\left(\widetilde{\B}_{\rig,K}^{\dagger}\right)^{K-\pa}\cdot\left(\log u\right)^{i}\\
 & =\oplus_{i=0}^{d}\left(\cup_{n\geq0}\varphi_{q}^{-n}\left(\B_{\rig,K}^{\dagger}\right)\right)\left(\log u\right)^{i}.
\end{aligned}
\]
Taking the colimit as $d\rightarrow\infty$ shows that $\left(\widetilde{\B}_{\log,K}^{\dagger}\right)^{K-\pa}=\left(\cup_{n\geq0}\varphi_{q}^{-n}\left(\B_{\rig,K}^{\dagger}\right)\right)[\log u]$.
It remains to show that $\cup_{n\geq0}\varphi_{q}^{-n}\left(\B_{\log,K}^{\dagger}\right)$
is contained $\left(\cup_{n\geq0}\varphi_{q}^{-n}\left(\B_{\rig,K}^{\dagger}\right)\right)[\log u]$,
as the inclusion in the other direction is obvious. Assume the opposite
and let $f=\sum_{i=0}^{d}a_{i}\left(\log u\right)^{i}$ be an element
of $\varphi_{q}^{-n}\left(\B_{\log,K}^{\dagger}\right)$ with $a_{i}\in\widetilde{\B}_{\rig,K}^{\dagger}$
and $d$ minimal such that $f$ is not contained in $\left(\cup_{n\geq0}\varphi_{q}^{-n}\left(\B_{\rig,K}^{\dagger}\right)\right)[\log u]$.
As $\varphi_{q}^{n}\left(f\right)\in\B_{\log,K}^{\dagger}$ and $\varphi_{q}(\log u)\equiv q\log u$
$\mod\B_{\rig,K}^{\dagger}$, examining the coefficient of $\left(\log u\right)^{d}$
reveals that $\varphi_{q}^{n}(a_{d})\in\B_{\rig,K}^{\dagger}$, providing
a contradiction. 
\end{proof}
We now proceed to prove Proposition 4.5. We do so in several steps,
following the method appearing in $\mathsection4$ of \cite{Be16}.
If $t\geq1$, we denote by $\mathrm{LA}_{t}(\mathcal{O}_{K})$ the
space functions of $\mathcal{O}_{K}$ which are analytic on closed
discs of radius $|\pi|^{t}$. For $a\in\mathcal{O}_{K}$, write $[a](T)=\sum_{n\geq1}c_{n}(a)T^{n}$.
Each $c_{n}(a)$ is a polynomial of degree at most $n$ in $a$, and
$c_{n}(\mathcal{O}_{K})\subset\mathcal{\mathcal{O}}_{K}$.
\begin{lem}
$\left|\left|c_{n}\right|\right|_{\mathrm{LA}_{t}}\leq|\pi|^{-\frac{n}{q^{t}(q-1)}}$.
\end{lem}

\begin{proof}
Recall that de Shalit has constructed in \cite{dS16} a Mahler basis
$\left\{ g_{n}(T)\right\} _{n\geq0}^{\infty}$ such that $g_{n}(T)$
is a polynomial in $K[T]$ of degree $n$ and and such that $||\pi^{w_{n,t}}g_{n}||_{\mathrm{LA}_{t}}=1$,
where $w_{n,t}$ is an integer satisfying $w_{n,t}\leq\frac{n}{q^{t}(q-1)}$.
As $c_{n}$ has degree at most $n$, we can write $c_{n}=\sum_{i=0}^{n}b_{n,i}g_{i}$
for some $b_{n,i}\in\mathcal{O}_{K}$, and so $\left|\left|c_{n}\right|\right|_{\mathrm{LA}_{t}}\leq\sup_{1\leq i\leq n}\left|\left|g_{i}\right|\right|\leq|\pi|^{-\frac{n}{q^{t}(q-1)}}$.
\end{proof}
Recall that $g(\log u)=\log u+\log\left(\frac{[a](u)}{u}\right)$,
where $a=\chi_{\pi}(g)$. We write 
\[
\log\left(\frac{[a](u)}{u}\right)=\sum_{n=1}^{\infty}d_{n}(a)u^{n}.
\]

\begin{lem}
$\left|\left|d_{n}\right|\right|_{\mathrm{LA}_{t}}\leq|\pi|^{-\frac{2n}{q^{t}(q-1)}+o(n)}$.
\end{lem}

\begin{proof}
Write $\frac{[a](u)}{u}-1=\sum_{n\geq0}e_{n}(a)u^{n}$, where $e_{0}(a)=a-1$
and $e_{n}(a)=c_{n+1}(a)$ for $n\geq1$. Then $d_{n}$ is a sum of
functions of the form
\[
\frac{\left(-1\right)^{m-1}}{m}\sum_{\substack{\left(k_{1},...,k_{m}\right)\in\Z_{\geq0}^{m}\\
k_{1}+...+k_{m}=n
}
}\prod_{i=1}^{m}e_{k_{i}},
\]
and it suffices to bound each such function by $|\pi|^{-\frac{2n}{q^{t}(q-1)}+o(n)}$,
where $o(n)$ does not depend on $m$.

Fix $\left(k_{1},...,k_{m}\right)\in\Z_{\geq0}^{m}$ with $k_{1}+...+k_{m}=n$.
Let $h$ be the number of $1\leq i\leq m$ such that $k_{i}\geq1$.
Then by Lemma 4.7 we have
\[
\left|\left|\prod_{i:1\leq k_{i}}e_{k_{i}}\right|\right|_{\mathrm{LA}_{t}}\leq\left|\pi\right|^{-\left(n+h\right)/q^{t}(q-1)}\leq|\pi|^{-2n/q^{t}(q-1)}.
\]
On the other hand, $\left|\left|e_{0}\right|\right|_{\mathrm{LA}_{t}}\leq\left|\pi\right|^{t}$,
so
\[
\left|\left|\frac{1}{m}\prod_{i:k_{i}=0}e_{k_{i}}\right|\right|_{\mathrm{LA}_{t}}\leq\left|\frac{1}{m}\right|\left|\pi\right|^{t(m-h)}\leq p^{\left[v_{p}(m)-t/e\max\left\{ 0,m-n\right\} \right]}=\left|\pi\right|^{o(n)}.
\]
Combining the two inequalities we obtain the claim.
\end{proof}
\emph{Proof of Proposition 4.5. }Write $r_{n}=p^{nf-1}(p-1)$ and
let $s\leq t$. It is enough to show that $\log u$ is $K$-analytic
on $\Gamma_{K_{t+1}}$ as a vector of $\widetilde{\B}_{K}^{[r_{s},r_{t}]}=\left(\widetilde{\B}^{[r_{l},r_{t}]}\right)^{H_{K}}$.
Since $g(\log u)=\log u+\log\left(\frac{[a](u)}{u}\right)$ for $a=\chi_{\pi}(g)$,
we need to verify that $\left|\left|d_{n}\right|\right|_{\mathrm{LA}_{t+1}}\left|\left|u^{n}\right|\right|_{[r_{s},r_{t}]}\rightarrow0$
as $n\rightarrow0$. By the maximum principle, we have $\left|\left|u\right|\right|_{[r_{s},r_{t}]}=\left|\left|u\right|\right|_{r_{t}}=\left|\pi\right|^{1/q^{t-1}(q-1)}$,
and so by Lemma 4.7
\[
\left|\left|d_{n}\right|\right|_{\mathrm{LA}_{t+1}}\left|\left|u^{n}\right|\right|_{[r_{s},r_{t}]}\leq|\pi|^{n\left[\frac{1}{q^{t-1}(q-1)}-\frac{2}{q^{t+1}(q-1)}\right]+o(1)},
\]
which approaches $0$, as required.$\hfill\ensuremath{\Box}$

\section{Lubin-Tate $(\varphi_{q},\Gamma_{K})$-modules }

In this section we recall how to attach a $(\varphi_{q},\Gamma_{K})$-module
over $\B_{\rig,K}^{\dagger}$ to a $K$-analytic $p$-adic representation
of $G_{K}$, and we express the invariants of $\mathsection3$ in
terms of these $\left(\varphi_{q},\Gamma_{K}\right)$-modules. Recall
from $\mathsection1.2$ that the field $E$ is a finite extension
of $K$ which serves as a field of coefficients.

\subsection{$K$-analytic $(\varphi_{q},\Gamma_{K})$-modules}

Let $\A_{K}$ be the ring of power series $f(T)=\sum_{k\in\Z}a_{k}T^{k}$
with $a_{k}\in\mathcal{O}_{K}$ such that $\val_{p}(a_{k})\rightarrow0$
as $k\rightarrow-\infty$, and let $\B_{K}=\A_{K}[1/\pi]$. The rings
$\A_{K}$ and $\B_{K}$ have a $\varphi_{q}$-action given by $\varphi_{q}(T)=\left[\pi\right](T)$
and a $\Gamma_{K}$-action given by $g(T)=\left[\chi_{\pi}(g)\right](T)$
for $g\in\Gamma_{K}$. These actions are $K$-linear, and we extend
them $E$-linearly to $\A_{K}\otimes_{K}E$ and $\B_{K}\otimes_{K}E$.
A $\left(\varphi_{q},\Gamma_{K}\right)$-module over $\B_{K}\otimes_{K}E$
is a finite free $\B_{K}\otimes_{K}E$ module $\D_{K}$ which has
commuting semilinear $\varphi_{q}$ and $\Gamma_{K}$ actions. We
say it is étale if there exists a basis of $\D_{K}$ for which $\Mat(\varphi)\in\GL_{d}(\A_{K}\otimes_{K}E)$. 

Kisin and Ren have shown in $\mathsection1$ of \cite{KR09} how to
associate to any $V\in\mathrm{Rep}_{E}(G_{K})$ an étale $\left(\varphi_{q},\Gamma_{K}\right)$-module
over $\B_{K}\otimes_{K}E$ which we denote $\D_{K}(V)$. Furthermore,
one has the following result.
\begin{thm}
The functor $V\mapsto\D_{K}(V)$ induces an equivalence of categories
\[
\left\{ E\text{-representations of }G_{K}\right\} \longleftrightarrow\left\{ \text{étale }\left(\varphi_{q},\Gamma_{K}\right)\text{-modules over }\B_{K}\otimes_{K}E\right\} .
\]
\end{thm}

Now let $\B_{K}^{\dagger}$ be the subring of $\B_{K}$ which consists
of power series which converge on some nonempty annulus $r\leq|T|<1$.
It is preserved by the $\left(\varphi_{q},\Gamma_{K}\right)$-structure,
and we say that a $\left(\varphi_{q},\Gamma_{K}\right)$-module over
$\B_{K}\otimes_{K}E$ is \emph{overconvergent }if $\D_{K}=\D_{K}^{\dagger}\otimes_{\B_{K}^{\dagger}}\B_{K}$
where $\D_{K}^{\dagger}$ is a $\left(\varphi_{q},\Gamma_{K}\right)$-module
over $\B_{K}^{\dagger}\otimes_{K}E$. A representation $V\in\mathrm{Rep}_{E}(G_{K})$
is said to be overconvergent if $\D_{K}(V)$ is. As $\B_{K}^{\dagger}$
is a subring of $\B_{\rig,K}^{\dagger}$, for such a $(\varphi_{q},\Gamma_{K})$-module
we can form $\D_{\rig,K}^{\dagger}=\D_{K}^{\dagger}\otimes_{\B_{K}^{\dagger}}\B_{\rig,K}^{\dagger}$
and $\D_{\log,K}^{\dagger}=\D_{K}^{\dagger}\otimes_{\B_{K}^{\dagger}}\B_{\log,K}^{\dagger}$. 

In the case $K=\Q_{p}$, Cherbonnier and Colmez have proven in \cite{CC98}
that $\D_{K}(V)$ is always overconvergent. Unfortunately, this is
no longer true whenever $K\neq\Q_{p}$ (see Theorem 0.6 of \cite{FX13}).
However, the analogue of the Cherbonnier-Colmez theorem does hold
for $K$-analytic representations, and, even better, we can characterize
the $(\varphi_{q},\Gamma_{K})$-modules which arise in this way. More
precisely, a $\left(\varphi_{q},\Gamma_{K}\right)$-module $\D_{\rig,K}^{\dagger}$
over $\B_{\rig,K}^{\dagger}\otimes_{K}E$ is called $K$-analytic
if $\D_{\rig,K}^{\dagger}=\left(\D_{\rig,K}^{\dagger}\right)^{K-\pa}$.
Then one has the following result (Theorems C and D of \cite{Be16}).
\begin{thm}
1. If $V\in\mathrm{Rep}_{E}(G_{K})$ is $K$-analytic, then $\D_{K}(V)$
is overconvergent.

2. The functor $V\mapsto\D_{\rig,K}^{\dagger}(V)$ gives an equivalence
of categories
\[
\left\{ K\text{-analytic }E\text{-linear representations of }G_{K}\right\} 
\]
\[
\longleftrightarrow\left\{ \text{étale }K\text{-analytic }\left(\varphi_{q},\Gamma_{K}\right)\text{-modules over }\B_{\rig,K}^{\dagger}\otimes_{K}E\right\} 
\]

3. If $V\in\mathrm{Rep}_{E}(G_{K})$ is $K$-analytic, then there
exists a natural $G_{K}$-equivariant isomorphism $\widetilde{\B}_{\rig}^{\dagger}\otimes_{K}V\cong\widetilde{\B}_{\rig}^{\dagger}\otimes_{\B_{\rig,K}^{\dagger}}\D_{\rig,K}^{\dagger}(V)$.
\end{thm}

All characters are overconvergent, so split $2$-dimensional representation
are always overconvergent. For nonsplit representations, Theorem 5.2
implies the following.
\begin{cor}
Let $V\in\mathrm{Rep}_{E}(G_{K})$ be a nonsplit $2$-dimensional
representation. The following are equivalent.

1. $V$ is overconvergent.

2. Either $V$ is $K$-analytic up to a character twist or $V$ is
an extension of the trivial representation by itself.
\end{cor}

\begin{proof}
 If $V\left(\delta\right)$ is $K$-analytic then it is overconvergent
by Theorem 5.2. In addition, Theorem 0.3 of \cite{FX13} shows that
every extension of the trivial representation by itself is overconvergent,
so (2) implies (1). In the converse direction, let $V$ be an overconvergent
representation. It is either absolutely irreducible or reducible nonsplit
after possibly extending scalars.

\textbf{Case 1: }$V$ is absolutely irreducible. Then Corollary 4.3
of \cite{Be13} implies that $V\left(\delta\right)$ is $K$-analytic
for some character $\delta$. Corollary 4.3 of \cite{Be13} is proved
there in the setting where $K$ is an unramified extension of $\Q_{p}$;
it is a consequence of Theorem 4.2 of ibid. This assumption can be
removed, because Theorem 4.2 of ibid is reproven in \cite{Be16} without
assuming $K$ is unramified. 

\textbf{Case 2: }$V$ is reducible nonsplit after extending scalars.
By the following lemma, extending scalars does not matter for the
question of oveconvergence. Thus, we may assume $V$ is reducible
and nonsplit, and after performing a character twist we may further
assume it is an extension of $1$ by $E(\delta)$ with $\det\left(V\right)=\delta$.
If $\delta=1$, we are done. Otherwise, by Theorem 0.4 of \cite{FX13}
a nontrivial overconvergent extension of $1$ by $E\left(\delta\right)$
can only occur if $\delta$ is $K$-analytic, and since $\delta\neq1$
this implies that $V$ itself is $K$-analytic by Theorem 0.3 of \cite{FX13}.
\end{proof}
The following lemma was used in the proof of Corollary 5.3.
\begin{lem}
Let $V\in\mathrm{Rep}_{E}(G_{K})$ be a representation, and let $E'$
be a finite extension of $E$. Then $V$ is overconvergent if and
only if $V\otimes_{E}E'$ is overconvergent.
\end{lem}

\begin{proof}
Clearly, if $V$ is overconvergent, so is $V\otimes_{E}E'$. This
being said, when proving the converse direction we are free to enlarge
$E'$, and in particular we may assume $E'/E$ is Galois. Write $\D_{E'}=\D_{\rig,K}^{\dagger}(V\otimes_{E}E')$,
which is an étale $(\varphi_{q},\Gamma_{K})$-module over $\B_{\rig,K}^{\dagger}\otimes_{K}E'$.
By Galois descent, $\D=\D_{E'}^{\Gal(E'/E)}$ is an étale $(\varphi_{q},\Gamma_{K})$-module
over $\B_{\rig,K}^{\dagger}\otimes_{K}E$. An explicit description
of the inverse functor to $\D_{\rig,K}^{\dagger}$ (provided for example
in Proposition 1.5 of \cite{FX13}) reveals that $\D$ corresponds
to $V$ under the equivalence of Theorem 5.2, because the $\Gal(E'/E)$-action
commutes with all the structure involved. This shows that $V$ is
overconvergent.
\end{proof}

\subsection{The modules $\protect\D_{*,K}$ and the extended dictionary}

Recall that for $n\geq0$ we set $r_{n}=p^{nf-1}(p-1)$. For $r>0$
we let $n(r)$ be the minimal $n$ such that $r_{n}\geq r$. If $I$
is a closed interval and $r_{0}=\frac{p-1}{p}\in I$, then for $\widetilde{\B}^{I}$
as in $\mathsection4$ the usual completion map $\widetilde{\A}^{+}\rightarrow\B_{\dR}^{+}$
extends to a map $\iota_{0}:\widetilde{\B}^{I}\rightarrow\B_{\dR}^{+}$.
More generally if $r_{n}\in I$ then one has the map $\iota_{n}=\iota_{0}\circ\varphi_{q}^{-n}:\widetilde{\B}^{I}\rightarrow\B_{\dR}^{+}$.
Now let $\B_{\rig,K}^{\dagger,r}=\widetilde{\B}_{\rig,K}^{\dagger,r}\cap\B_{\rig,K}^{\dagger}$,
then for $n\geq n(r)$ the map above restricts to give $\iota_{n}:\B_{\rig,K}^{\dagger,r}\rightarrow K_{n}\left[\left[t_{K}\right]\right]\subset\B_{\dR}$.
By Théorème I.3.3 of \cite{Be08}, if $r\gg0$, there exists a unique
$\B_{\rig,K}^{\dagger,r}$-submodule $\D_{\rig,K}^{\dagger,r}$ of
$\D_{\rig,K}^{\dagger}$ such that $\D_{\rig,K}^{\dagger}=\D_{\rig,K}^{\dagger,r}\otimes_{\B_{\rig,K}^{\dagger,r}}\B_{\rig,K}^{\dagger}$
and such that the $\B_{\rig,K}^{\dagger,pr}$-module $\B_{\rig,K}^{\dagger,pr}\otimes_{\B_{\rig,K}^{\dagger,r}}\D_{\rig,K}^{\dagger,r}$
has a basis contained in $\varphi_{q}\left(\D_{\rig,K}^{\dagger,r}\right)$.
Finally, let $t_{K}=\log_{\mathcal{F}}(T)\in\B_{\rig,K}^{\dagger}$;
it belongs to $\B_{\rig,K}^{\dagger,r_{0}}$ and $\iota_{0}(t_{K})$
coincides with the usual $t_{K}$ of $\B_{\dR}^{+}$ as in $\mathsection2$
and $\mathsection3$. We set

\[
\begin{aligned} & \mathrm{D}_{\mathrm{Sen},K}\left(\mathrm{D}_{\mathrm{rig},K}^{\dagger}\right)=\left(D_{\mathrm{rig},K}^{\dagger,r}\otimes_{\theta\circ\varphi_{q}^{-n}}K_{n}\right)\otimes_{K_{n}}K_{\infty},\\
 & \D_{\dif,K}\left(\D_{\rig,K}^{\dagger}\right)=\left(\D_{\rig,K}^{\dagger,r}\otimes_{\iota_{n}}K_{n}\left(\left(t_{K}\right)\right)\right)\otimes_{K_{n}}K_{\infty}\left(\left(t_{K}\right)\right),\\
 & \D_{\dR,K}\left(\D_{\rig,K}^{\dagger}\right)=\D_{\dif,K}\left(\D_{\rig,K}^{\dagger}\right)^{\Gamma_{K}},\\
 & \D_{\cris,K}\left(\D_{\rig,K}^{\dagger}\right)=\left(\D_{\rig,K}^{\dagger}\left[1/t_{K}\right]\right)^{\Gamma_{K}},\\
 & \D_{\st,K}\left(\D_{\rig,K}^{\dagger}\right)=\left(\D_{\log,K}^{\dagger}\left[1/t_{K}\right]\right)^{\Gamma_{K}}.
\end{aligned}
\]

One verifies that $\D_{\Sen,K}\left(\D_{\rig,K}^{\dagger}\right)$
and $\D_{\dif,K}\left(\D_{\rig,K}^{\dagger}\right)$ are independent
of the choice of $n$. The main theorem of this section is the following.
\begin{thm}
Let $V$ be $K$-analytic representation of $G_{K}$. For $*\in\left\{ \Sen,\dif,\dR,\cris,\st\right\} $,
we have a natural isomorphism
\[
\D_{*,K}(V)\cong\D_{*,K}\left(\D_{\rig,K}^{\dagger}(V)\right).
\]
\end{thm}

\begin{proof}
Set $\D_{\rig,K}^{\dagger}=\D_{\rig,K}^{\dagger}(V)$. For $r,n\gg0$,
we have a natural map $\D_{\rig,K}^{\dagger,r}\xrightarrow{\theta\circ\varphi_{q}^{-n}}W$,
where $W=\left(\C_{p}\otimes_{K}V\right)^{H_{K}}$. The image of $\theta\circ\varphi_{q}^{-n}$
is by definition $\D_{\Sen,K}\left(\D_{\rig,K}^{\dagger}\right)$,
which is a $K_{\infty}$-submodule of rank $d=\dim_{K}V$. As $\theta\circ\varphi_{q}^{-n}$
is $\Gamma_{K}$ equivariant, it maps pro $K$-analytic vectors to
locally $K$-analytic vectors, so the image lands in $W^{K-\la}=\D_{\Sen,K}(V)$.
Comparing ranks we get the desired isomorphism for $*=\mathrm{Sen}$\@.
Replacing $\theta\circ\varphi_{q}^{-n}$ by $\iota_{0}\circ\varphi_{q}^{-n}$
we similarly get a map $\D_{\dif,K}^{+}\left(V\right)\rightarrow\D_{\dif,K}^{+}\left(\D_{\rig,K}^{\dagger,r}\right)$
of two $K_{\infty}\left[\left[t_{K}\right]\right]$-modules of rank
$d$, whose reduction mod $t_{K}$ is the isomorphism $\D_{\Sen,K}(V)\xrightarrow{\sim}\D_{\Sen,K}\left(\D_{\rig,K}^{\dagger}\right)$.
Thus by Nakayama's lemma we have $\D_{\dif,K}^{+}(V)\cong\D_{\dif,K}^{+}\left(\D_{\rig,K}^{\dagger}(V)\right)$
and we deduce $\D_{\dif,K}(V)\cong\D_{\dif,K}\left(\D_{\rig,K}^{\dagger}(V)\right)$
by passing to colimits.

As $\D_{\cris,K}=\D_{\st,K}^{N=0}$, it remains to prove the comparison
for $*=\st$. Twisting $V$ by an appropriate power of $\chi_{\pi}$,
we further reduce to proving that $\D_{\st,K}^{+}\left(V\right)=\left(\D_{\log,K}^{\dagger}\right)^{\Gamma_{K}}$.
By Lemma 5.6 below, we have

\[
\begin{aligned}\D_{\st,K}^{+}\left(V\right) & =\left(\widetilde{\B}_{\log}^{\dagger}\otimes_{K}V\right)^{G_{K}}\\
 & =\left(\widetilde{\B}_{\log}^{\dagger}\otimes_{\B_{\log,K}^{\dagger}}\D_{\log,K}^{\dagger}(V)\right)^{G_{K}}\\
 & =\left(\widetilde{\B}_{\log,K}^{\dagger}\otimes_{\B_{\log,K}^{\dagger}}\D_{\log,K}^{\dagger}(V)\right)^{\Gamma_{K}}.
\end{aligned}
\]
On the one hand, this implies that $\D_{\st,K}^{+}\left(V\right)\subset\left(\D_{\log,K}^{\dagger}\right)^{\Gamma_{K}}$.
On the other hand, vectors which are fixed by $\Gamma_{K}$ are also
pro $K$-analytic on $\Gamma_{K}$, so

\[
\begin{aligned}\D_{\st,K}^{+}\left(V\right) & =\left(\widetilde{\B}_{\log,K}^{\dagger}\otimes_{\B_{\log,K}^{\dagger}}\D_{\log,K}^{\dagger}(V)\right)^{\Gamma_{K},K-\pa}\\
 & =\left(\left(\widetilde{\B}_{\log,K}^{\dagger}\otimes_{\B_{\log,K}^{\dagger}}\D_{\log,K}^{\dagger}(V)\right)^{K-\pa}\right)^{\Gamma_{K}}.
\end{aligned}
\]
Since $V$ is $K$-analytic, Theorem 5.2 implies that $\D_{\log,K}^{\dagger}(V)$
is pro $K$-analytic, and so by Lemma 2.1 we have 
\[
\left(\widetilde{\B}_{\log,K}^{\dagger}\otimes_{\B_{\log,K}^{\dagger}}\D_{\log,K}^{\dagger}(V)\right)^{K-\pa}=\left(\widetilde{\B}_{\log,K}^{\dagger}\right)^{K-\pa}\otimes_{\B_{\log,K}^{\dagger}}\D_{\log,K}^{\dagger}(V).
\]
Applying Theorem 4.6, we deduce
\[
\D_{\st,K}^{+}\left(V\right)\subset\left(\cup_{n\geq0}\varphi_{q}^{-n}\left(\B_{\log,K}^{\dagger}\right)\otimes_{\B_{\log,K}^{\dagger}}\D_{\log,K}^{\dagger}(V)\right)^{\Gamma_{K}}.
\]
Thus $\varphi_{q}^{n}\left(\D_{\st,K}^{+}\left(V\right)\right)\subset\left(\D_{\log,K}^{\dagger}\right)^{\Gamma_{K}}$
for some $n\gg0$. If $e_{1},...,e_{l}$ is a basis of $\D_{\st,K}^{+}\left(V\right)$
then $\varphi_{q}^{n}(e_{1}),...,\varphi_{q}^{n}(e_{l})$ gives another
basis of $\D_{\st,K}^{+}\left(V\right)$ which lies in $\left(\D_{\log,K}^{\dagger}\right)^{\Gamma_{K}}$.
This concludes the proof.
\end{proof}
The following lemma was used in the proof of Theorem 5.5.
\begin{lem}
Let $V\in\mathrm{Rep}_{E}(G_{K})$. Then 
\[
\D_{\st,K}^{+}(V)=\left(\widetilde{\B}_{\log}^{\dagger}\otimes_{K}V\right)^{G_{K}}.
\]
\end{lem}

\begin{proof}
Given an automorphism $\sigma:K_{0}\rightarrow K_{0}$, let $V^{\sigma}$
be the $\sigma$-twist of $V$. Then we have $G_{K}$-compatible identifications
$\B_{\st}^{+}\otimes_{\Q_{p}}V\cong\oplus_{\sigma}\B_{\st}^{+}\otimes_{K_{0}}V^{\sigma}$
and $\widetilde{\B}_{\log,0}^{\dagger}\otimes_{\Q_{p}}V\cong\oplus_{\sigma}\widetilde{\B}_{\log,0}^{\dagger}\otimes_{K_{0}}V^{\sigma}$.
Now by Proposition 3.4 of \cite{Be02} we have $\left(\B_{\st}^{+}\otimes_{\Q_{p}}V\right)^{G_{K}}=\left(\widetilde{\B}_{\log,0}^{\dagger}\otimes_{\Q_{p}}V\right)^{G_{K}}$
and hence by projecting to the $\sigma=\Id$ component
\[
\D_{\st,K}^{+}(V)=\left(\B_{\st}^{+}\otimes_{K_{0}}V\right)^{G_{K}}=\left(\widetilde{\B}_{\log,0}^{\dagger}\otimes_{K_{0}}V\right)^{G_{K}}.
\]
Finally, by Proposition 4.3 we have $\widetilde{\B}_{\log}^{\dagger}=\widetilde{\B}_{\log,0}^{\dagger}\otimes_{K_{0}}K$
so $\left(\widetilde{\B}_{\log,0}^{\dagger}\otimes_{K_{0}}V\right)^{G_{K}}=\left(\widetilde{\B}_{\log}^{\dagger}\otimes_{K}V\right)^{G_{K}}$. 
\end{proof}
\begin{rem}
The definitions given in this section for $\D_{\Sen,K},\D_{\dif,K},\D_{\dR,K},\D_{\cris,K},\D_{\st,K}$
make sense for non étale $K$-analytic $(\varphi_{q},\Gamma_{K})$-modules.
The properties of these modules which were proved in $\mathsection3$
carry over with no difficulty to this more general case.
\end{rem}

\begin{rem}
For each $\D_{\rig,K}^{\dagger}$, we define filtrations on $\D_{\dR,K}\left(\D_{\rig,K}^{\dagger}\right)$
and $\D_{\cris,K}\left(\D_{\rig,K}^{\dagger}\right)$. For $i\in\Z$,
we set
\[
\Fil^{i}\left(\D_{\dR,K}\left(\D_{\rig,K}^{\dagger}\right)\right)=\left(K_{\infty}\otimes_{K_{n}}t_{K}^{i}K_{n}\left[\left[t_{K}\right]\right]\otimes_{\B_{\rig,K}^{\dagger,r}}\D_{\rig,K}^{\dagger,r}\right)^{\Gamma_{K}}
\]

for $n\geq n(r)$. Recall there are injections $\iota_{n}:\D_{\cris,K}\left(\D_{\rig,K}^{\dagger}\right)\hookrightarrow\D_{\dR,K}\left(\D_{\rig,K}^{\dagger}\right)$
for $n\gg0$, and set

\[
\varphi_{q}^{n}\left(\Fil^{i}\D_{\cris,K}\left(\D_{\rig,K}^{\dagger}\right)\right)=\iota_{n}^{-1}\left(\Fil^{i}\left(\D_{\dR,K}\left(\D_{\rig,K}^{\dagger}\right)\right)\right).
\]
The filtrations $\Fil^{i}\left(\D_{\dR,K}\left(\D_{\rig,K}^{\dagger}\right)\right)$
and $\Fil^{i}\left(\D_{\cris,K}\left(\D_{\rig,K}^{\dagger}\right)\right)$
are independent of all choices. If $\D=\D_{\rig,K}^{\dagger}(V)$,
one checks that the usual filtration of $\D_{\dR,K}\left(V\right)$
(resp. $\D_{\cris,K}(V)$) induced from that of $\B_{\dR}$ (resp.
$\B_{\max,K}$) coincides with the above defined filtration via the
identification of Theorem 5.5.
\end{rem}

\section{Lubin-Tate trianguline representations of dimension 2}

We keep the convention that $E$ is a finite extension of $\Q_{p}$
which contains $K$.
\begin{defn}
1. A $\left(\varphi_{q},\Gamma_{K}\right)$-module over $\B_{\rig,K}^{\dagger}\otimes_{K}E$
is called Lubin-Tate trianguline if it can be written as a successive
extension of $\left(\varphi_{q},\Gamma_{K}\right)$-modules of rank
1.

2. An $E$-linear $K$-analytic representation $V$ is called Lubin-Tate
trianguline if $\D_{\rig,K}^{\dagger}(V)$ is Lubin-Tate trianguline.
\end{defn}

In the case $K=\Q_{p}$, trianguline $\left(\varphi_{q},\Gamma_{K}\right)$-modules
of dimension 2 were first studied by Colmez in \cite{Co08}. We shall
be concerned with Lubin-Tate trianguline $\left(\varphi_{q},\Gamma_{K}\right)$-modules
of dimension 2 which were studied by Fourquaux and Xie in \cite{FX13}.

\subsection{Characters of the Weil group}

Recall that if $W_{K}$ is the Weil group of $K$, local class field
theory gives a natural isomorphism $W_{K}^{\mathrm{ab}}\cong K^{\times}$.
This allows us to identify characters $\delta:K^{\times}\rightarrow E^{\times}$
with characters $W_{K}^{\mathrm{ab}}\rightarrow E^{\times}$, the
identification given by 
\[
\delta(\mathrm{Frob}_{\pi}^{-n}g)=\delta(\pi)^{n}\delta(\chi_{\pi}(g))
\]
 for $g\in\Gal(K^{\mathrm{ab}}/K^{\mathrm{un}})$ and $n\in\Z$. To
such characters $\delta$ we associate the $\left(\varphi_{q},\Gamma_{K}\right)$-module
$\left(\B_{\rig,K}^{\dagger}\otimes_{K}E\right)(\delta)$. It is a
$\left(\varphi_{q},\Gamma_{K}\right)$-module of rank 1 with a basis
$e_{\delta}$, where $\varphi_{q}(e_{\delta})=\delta(\pi)e_{\delta}$
and $g(e_{\delta})=\delta(\chi_{\pi}(g))$ for $g\in\Gamma_{K}$.
Note that this $\left(\varphi_{q},\Gamma_{K}\right)$-module is étale
exactly when $\delta$ is unitary; in this case, if $\delta$ is locally
$K$-analytic, the module $\left(\B_{\rig,K}^{\dagger}\otimes_{K}E\right)(\delta)$
corresponds under the equivalence of categories in $\mathsection5.1$
to the extension of $\delta$ to $\Gal(\overline{K}/K)$. Proposition
1.9 of \cite{FX13} shows that all $K$-analytic $(\varphi_{q},\Gamma_{K})$-modules
of rank 1 over $\B_{\rig,K}^{\dagger}\otimes_{K}E$ are obtained in
this way. We write $\mathscr{I}_{\mathrm{an}}=\mathscr{I}_{\mathrm{an}}(E)$
for the set of locally $K$-analytic Weil characters. There are two
characters in $\mathscr{I}_{\mathrm{an}}$ of particular interest:
the inclusion character $x:K^{\times}\rightarrow E^{\times}$ and
the character $\mu_{\lambda}(z)=\lambda^{\val_{\pi}\left(z\right)}$
. To $\delta\in\mathscr{I}_{\mathrm{an}}$ we associate to the weight
$w(\delta)=\frac{\log_{p}\delta(u)}{\log_{p}u}$ where $u\in\mathcal{O}_{K}^{\times}$
is any element with $\log_{p}u\neq0$, and then $w(\delta)$ does
not depend on $u$. If $\delta$ is unitary and $w(\delta)\in\Z$
then $w(\delta)$ is the $K$-Hodge-Tate weight of the associated
character of $\Gal(\overline{K}/K)$ (see $\mathsection3$).

\subsection{Extensions}

Given $\delta_{1},\delta_{2}\in\mathscr{I}_{\mathrm{an}}$ we consider
the set of extensions 
\[
0\rightarrow\left(\B_{\rig,K}^{\dagger}\otimes_{K}E\right)(\delta_{1})\rightarrow\D_{\rig,K}^{\dagger}\rightarrow\left(\B_{\rig,K}^{\dagger}\otimes_{K}E\right)(\delta_{2})\rightarrow0
\]
in the category of $K$-analytic $\left(\varphi_{q},\Gamma_{K}\right)$-modules.
These extensions are classified by a finite-dimensional $E$-vector
space $\H_{\an}^{1}\left(\delta_{1}\delta_{2}^{-1}\right)$, whose
dimension is determined in Theorem 0.3 of \cite{FX13} as follows.
\begin{thm}
$\dim_{E}\H_{\an}^{1}\left(\delta_{1}\delta_{2}^{-1}\right)=2$ if
$\delta_{1}\delta_{2}^{-1}=x^{-i}$ for $i\in\Z_{\geq1}$ or if $\delta_{1}\delta_{2}^{-1}=\mu_{q^{-1}}x^{i}$
for $i\in\Z_{\geq0}$. Otherwise, $\dim_{E}\H_{\an}^{1}\left(\delta_{1}\delta_{2}^{-1}\right)=1$.
\end{thm}

\subsection{Spaces of Lubin-Tate trianguline $\left(\varphi_{q},\Gamma_{K}\right)$-modules
of dimension 2}

There is an action of $\mathbb{G}_{m}(E)=E^{\times}$ on $\H_{\an}^{1}\left(\delta_{1}\delta_{2}^{-1}\right)$,
and extensions which lie in the same orbit of this action give rise
to isomorphic $\left(\varphi_{q},\Gamma_{K}\right)$-modules. Following
$\mathsection6$ of \cite{FX13} we write 
\[
\mathscr{S}^{\an}=\mathscr{S}^{\an}(E)=\left\{ s=(\delta_{1},\delta_{2},\mathcal{L}):\delta_{1},\delta_{2}\in\mathscr{I}_{\mathrm{an}}\left(E\right),\mathcal{L}\in\H_{\an}^{1}(\delta_{1}\delta_{2}^{-1})\backslash\left\{ 0\right\} \right\} /\mathbb{G}_{m}(E).
\]
By Theorem 6.2, each pair of characters $\delta_{1},\delta_{2}\in\mathscr{I}_{\mathrm{an}}$
give rise either to a unique point $(\delta_{1},\delta_{2},\infty)$
of $\mathscr{S}^{\an}$ in the generic case or a $\P^{1}(E)$-family
of points of $\mathscr{S}^{\an}$ in the non generic case. To each
such $s$ we associate the correspnding $\left(\varphi_{q},\Gamma_{K}\right)$-module
$\D_{\rig,K}^{\dagger}(s)$ which is an extension of $\left(\B_{\rig,K}^{\dagger}\otimes_{K}E\right)(\delta_{1})$
by $\left(\B_{\rig,K}^{\dagger}\otimes_{K}E\right)(\delta_{2})$. 

Inside $\mathscr{S}^{\an}$, we consider the subset $\mathscr{S}_{+}^{\an}$
of real interest given by these $s\in\mathscr{S}^{\an}$ such that
\[
\val_{\pi}(\delta_{1}\delta_{2})=0\text{ and }\val_{\pi}(\delta_{1}(\pi))\geq0.
\]
For example, all étale $K$-analytic Lubin-Tate trianguline $(\varphi_{q},\Gamma_{K})$-modules
appear in $\mathscr{S}_{+}^{\an}$.

For an element $s\in\mathscr{S}_{+}^{\an}$ we associate two invariants:
the slope $u(s)=\val_{\pi}(\delta_{1}(\pi))$ and the weight $w(s)=w(\delta_{1}\delta_{2}^{-1})=w(\delta_{1})-w(\delta_{2})$.
We then have the following partition 
\[
\mathscr{S}_{+}^{\an}=\mathscr{S}_{+}^{\mathrm{ng}}\amalg\mathscr{S}_{+}^{\cris}\amalg\mathscr{S}_{+}^{\mathrm{\st}}\amalg\mathscr{S}_{+}^{\mathrm{\mathrm{ord}}}\amalg\mathscr{S}_{+}^{\mathrm{ncl}},
\]
where
\[
\begin{aligned} & \mathscr{S}_{+}^{\mathrm{ng}}=\left\{ s\in\mathscr{S}_{+}^{\an}:w(s)\notin\Z_{\geq1}\right\} ,\\
 & \mathscr{S}_{+}^{\cris}=\left\{ s\in\mathscr{S}_{+}^{\an}:w(s)\in\Z_{\geq1},u(s)<w(s),\mathcal{L}=\infty\right\} ,\\
 & \mathscr{S}_{+}^{\mathrm{\st}}=\left\{ s\in\mathscr{S}_{+}^{\an}:w(s)\in\Z_{\geq1},u(s)<w(s),\mathcal{L}\neq\infty\right\} ,\\
 & \mathscr{S}_{+}^{\mathrm{\mathrm{ord}}}=\left\{ s\in\mathscr{S}_{+}^{\an}:w(s)\in\Z_{\geq1},u(s)=w(s)\right\} ,\\
 & \mathscr{S}_{+}^{\mathrm{ncl}}=\left\{ s\in\mathscr{S}_{+}^{\an}:w(s)\in\Z_{\geq1},u(s)>w(s)\right\} .
\end{aligned}
\]

(Here $\mathrm{ng}$ and $\mathrm{ncl}$ are abbreviations for ``non-geometric''
and ``non-classical''). We also write $\mathscr{S}_{0}^{\an}=\left\{ s\in\mathscr{S}_{+}^{\an}:u(s)=0\right\} $
and $\mathscr{S}_{0}^{*}=\mathscr{S}_{+}^{*}\cap\mathscr{S}_{0}^{\an}$.
Each subset $\mathscr{S}_{+}^{*}$ above is named according to the
behaviour that the $s\in\mathscr{S}_{+}^{*}$ exhibit. For example,
$s\in\mathscr{S}_{+}^{\an}$ is étale if and only if $s\notin\mathscr{S}_{+}^{\mathrm{ncl}}$,
and in that case if $s\in\mathscr{S}_{+}^{\cris}$ (resp. $s\in\mathscr{S}_{+}^{\mathrm{\st}}$,
$s\in\mathscr{S}_{+}^{\mathrm{\mathrm{ord}}}$) then $\D_{\rig,K}^{\dagger}(s)$
comes from a potentially crystalline (resp. semistable but non-crystalline,
potentially ordinary) $E$-representation up to a twist. See $\mathsection6$
of \cite{FX13} for more details.
\begin{lem}
Let $\D_{\rig,K}^{\dagger}$ be a $K$-analytic $(\varphi_{q},\Gamma_{K})$-module
over $\B_{\rig,K}^{\dagger}\otimes_{K}E$. Then 
\[
\nabla\left(t_{K}\D_{\dif,K}^{+}\left(\D_{\rig,K}^{\dagger}\right)\right)\subset t_{K}\D_{\dif,K}^{+}\left(\D_{\rig,K}^{\dagger}\right).
\]
.
\end{lem}

\begin{proof}
Use the identity $\nabla(t_{K}x)=t_{K}x+t_{K}\nabla(x)=t_{K}(x+\nabla(x))$.
\end{proof}
The following shows that if $s\in\mathscr{S}_{+}^{\an}\backslash\left(\mathscr{S}_{+}^{\mathrm{ng}}\cup\mathscr{S}_{+}^{\mathrm{ncl}}\right)$
then $\D_{\rig,K}^{\dagger}(s)$ is comes from a de Rham $E$-representation
up to a twist (see Corollary 3.10).
\begin{cor}
Let $s=(\delta_{1},\delta_{2},\mathcal{L})\in\mathscr{S}_{+}^{\an}$
and suppose that $w(\delta_{1}),w(\delta_{2})\in\Z$ with $w(\delta_{1})>w(\delta_{2})$.
Then $\D_{\rig,K}^{\dagger}(s)$ is $K$-de Rham.
\end{cor}

\begin{proof}
We may assume that $\delta_{1}=1$. Write $\delta=\delta_{2}$ so
that and $w(\delta)<0$. Then $\D_{\dif,K}^{+}=\D_{\dif,K}^{+}\left(\D_{\rig,K}^{\dagger}\right)$
is an extension of the form 
\[
0\rightarrow K_{\infty}[[t_{K}]]\otimes_{K}E\rightarrow\D_{\dif,K}^{+}\rightarrow\D_{\dif,K}^{+}\left(\left(\B_{\rig,K}^{\dagger}\otimes_{K}E\right)(\delta)\right)\rightarrow0.
\]
Take $e_{\delta}\in\D_{\dR,K}^{+}\left(\left(\B_{\rig,K}^{\dagger}\otimes_{K}E\right)(\delta)\right)=\D_{\dif,K}^{+}\left(\left(\B_{\rig,K}^{\dagger}\otimes_{K}E\right)(\delta)\right)^{\Gamma_{K}=1}$
and lift it to $\D_{\dif,K}^{+}$. If we take $e=1\in K_{\infty}[[t_{K}]]\otimes_{K}E$
then $e,e_{\delta}$ is a basis of $\D_{\dif,K}^{+}$, and the action
of $\nabla_{\dif,K}$ on $\D_{\dif,K}^{+}$ in the basis $e,e_{\delta}$
is given by
\[
\Mat(\nabla_{\dif,K})=\left(\begin{array}{cc}
0 & f\\
0 & 0
\end{array}\right),
\]
where $f\in K_{\infty}[[t_{K}]]\otimes_{K}E$. As $w(\delta)<0$,
we have that $e_{\delta}$ is divisible by $t_{K}$ so by Lemma 6.3
we have that $f$ is divisible by $t_{K}$. As $\nabla_{\dif,K}=t_{K}\frac{\partial}{\partial t_{K}}$
on $K_{\infty}[[t_{K}]]$ we may find an $h\in K_{\infty}[[t_{K}]]\otimes_{K}E$
with $\nabla_{\dif,K}(h)=f$. Then $e$ and $e_{\delta}-he$ give
a full set of sections for $\nabla_{\dif,K}$ on $\D_{\dif,K}^{+}$.
By Proposition 3.7 (which applies according to Remark 5.6) we are
done.
\end{proof}
\begin{rem}
The argument of Corollary 6.4 can be generalized to show that if $\D_{\rig,K}^{\dagger}$
is a $K$-analytic Lubin-Tate trianguline $\left(\varphi_{q},\Gamma_{K}\right)$-module
which is a successive extension of $\left(\B_{\rig,K}^{\dagger}\otimes_{K}E\right)(\delta_{i})$
for $i=1,...,d$ with $w(\delta_{i})$ integers satisfying $w(\delta_{1})>...>w(\delta_{d})$,
then $\D_{\rig,K}^{\dagger}$ is $K$-de Rham.
\end{rem}

\subsection{Lubin-Tate triangulation of a $p$-adic representation of dimension
2}

The following generalizes Lemma 2.4.2 of \cite{BC09}. Using Theorem
5.5, we may follow the same proof of loc. cit.
\begin{prop}
Let $\D_{\rig,K}^{\dagger}$ be a $K$-analytic $\left(\varphi_{q},\Gamma_{K}\right)$-module
over $\B_{\rig,K}^{\dagger}\otimes_{K}E$ , $\alpha\in E^{\times}$
and $i\in\Z$. If $x\in\D_{\cris,K}\left(\D_{\rig,K}^{\dagger}\right)^{\varphi_{q}=\alpha}\subset\D_{\rig,K}^{\dagger}[1/t_{K}]$
, then $x\in\Fil^{i}\left(\D_{\cris,K}\left(\D_{\rig,K}^{\dagger}\right)\right)$
if and only if $x\in t_{K}^{i}\D_{\rig,K}^{\dagger}$.
\end{prop}

\begin{proof}
We may reduce to the case $i=0$ by twisting. By Theorem 5.5, $\Fil^{0}\left(\D_{\rig,K}^{\dagger}\right)=\D_{\cris,K}^{+}\left(\D_{\rig,K}^{\dagger}\right)=\left(\D_{\rig,K}^{\dagger}\right)^{\Gamma_{K}}$
and $\D_{\dR,K}^{+}\left(\D_{\rig,K}^{\dagger}\right)=\D_{\dif,K}^{+}\left(\D_{\rig,K}^{\dagger}\right)^{\Gamma_{K}}$.
Recall that for $n\gg0$ we have an injection
\[
\iota_{n}:\D_{\cris,K}\left(\D_{\rig,K}^{\dagger}\right)\hookrightarrow\D_{\dR,K}\left(\D_{\rig,K}^{\dagger}\right)
\]
and by the definition of filtration in Remark 5.8, we have

\[
\varphi_{q}^{n}\left(\D_{\cris,K}^{+}\left(\D_{\rig,K}^{\dagger}\right)\right)=\iota_{n}^{-1}\left(\D_{\dR,K}^{+}\left(\D_{\rig,K}^{\dagger}\right)\right).
\]

If $x\in\D_{\rig,K}^{\dagger}\cap\D_{\cris,K}\left(\D_{\rig,K}^{\dagger}\right)^{\varphi_{q}=\alpha}$
we have $\iota_{n}(x)\in\iota_{n}\left(\D_{\rig,K}^{\dagger}\right)\subset\D_{\dR,K}^{+}\left(\D_{\rig,K}^{\dagger}\right)$
for $n\gg0$, so $x\in\varphi_{q}^{n}\left(\D_{\cris,K}^{+}\left(\D_{\rig,K}^{\dagger}\right)\right)$
and the relation $x=\alpha^{n}\varphi_{q}^{-n}(x)$ implies $x\in\D_{\cris,K}^{+}\left(\D_{\rig,K}^{\dagger}\right)$.
Conversely, if $x\in\D_{\cris,K}^{+}\left(\D_{\rig,K}^{\dagger}\right)$
satisfies $\varphi_{q}(x)=x\alpha$, then $x=\alpha^{-n}\varphi_{q}^{n}(x)$,
so $\iota_{n}\left(x\right)$ lies in $\D_{\dR,K}^{+}\left(\D_{\rig,K}^{\dagger}\right)$
for $n\gg0$. Choose a $\B_{\rig,K}^{\dagger,r}$-basis $e_{1},...,e_{d}$
of $\D_{\rig,K}^{\dagger,r}$ for $r\gg0$, and write $x=\sum f_{j}e_{j}$
for $f_{j}\in\B_{\rig,K}^{\dagger,r}[1/t_{K}]$. For each $j$, we
have that $\iota_{n}\left(f_{j}\right)$ lies in $K_{n}\left[\left[t_{K}\right]\right]$
for $n\gg0$. This implies by Lemma 4.6 of \cite{Be02} that each
$f_{j}\in\B_{\rig,K}^{\dagger,r}$, whence $x\in\D_{\rig,K}^{\dagger}$. 
\end{proof}
Following $\mathsection3$ of \cite{Ch08}, we compute the triangulation
of a representation of dimension $2$ in terms of a crystalline period.
\begin{prop}
Let $V$ be a 2-dimensional $E$-linear $K$-analytic representation
of $G_{K}$. Then $V$ is Lubin-Tate trianguline if and only if there
exists a $K$-analytic character $\eta:G_{K}\rightarrow\mathcal{O}_{E}^{\times}$
and $\alpha\in E^{\times}$ such that $\D_{\cris,K}(V(\eta))^{\varphi_{q}=\alpha}\neq0$.
Moreover, if $i$ is the largest integer such that $\Fil^{i}\D_{\cris,K}(V(\eta))^{\varphi_{q}=\alpha}\nsubseteq\Fil^{i+1}\D_{\cris,K}(V(\eta))^{\varphi_{q}=\alpha}$,
then $\D_{\rig,K}^{\dagger}(V)$ is an extension of $\left(\B_{\rig,K}^{\dagger}\otimes_{K}E\right)(\delta_{1})$
by $\left(\B_{\rig,K}^{\dagger}\otimes_{K}E\right)(\delta_{2})$ where
$\delta_{1}=\eta^{-1}\mu_{\alpha}x^{-i}$ and $\delta_{2}=\eta\mu_{\alpha^{-1}}x^{i}\det(V)$.
\end{prop}

\begin{proof}
If $V$ is Lubin-Tate trianguline, then $\D_{\mathrm{rig,K}}^{\dagger}(V)$
contains a submodule of rank 1 isomorphic to $\left(\B_{\rig,K}^{\dagger}\otimes_{K}E\right)(\delta)$
for some $\delta\in\mathscr{I}_{\mathrm{an}}$. Taking $\eta:G_{K}\rightarrow\mathcal{O}_{E}^{\times}$
defined by $\eta(g)=\delta^{-1}(\chi_{\pi}(g))$ we have $\D_{\cris,L}^{+}(V(\eta))^{\varphi_{q}=\delta(\pi)}=\D_{\mathrm{rig}}^{\dagger}(V(\eta))^{\Gamma_{L}=1,\varphi_{q}=\delta(\pi)}\neq0$.
Conversely, suppose that such an $\alpha$ and $\eta$ exist. We shall
shows $V$ is Lubin-Tate trianguline with the described triangulation.
Twisting by a power of $\chi_{\pi}$, we may assume that $i=0$ and
that $\D_{\cris,K}^{+}(V(\eta))^{\varphi_{q}=\alpha}$ contains an
element $f\notin\Fil^{1}\D_{\cris,K}(V(\eta))^{\varphi_{q}=\alpha}$.
By what we have proven in $\mathsection5$, we have
\[
\D_{\cris,K}^{+}(V(\eta))^{\varphi_{q}=\alpha}=\D_{\mathrm{rig},K}^{\dagger}(V(\eta))^{\Gamma_{K}=1,\varphi_{q}=\alpha},
\]
so $f\in\D_{\mathrm{rig},K}^{\dagger}(V(\eta))^{\Gamma_{K}=1,\varphi_{q}=\alpha}$.
By taking its span and twisting by $\eta^{-1}$ we get a rank 1 sub
$\left(\varphi_{q},\Gamma_{K}\right)$-module of $\D_{\mathrm{rig},K}^{\dagger}(V)$.
The ideal $I$ generated by the coefficients of $f$ in a basis of
in $\D_{\mathrm{rig},K}^{\dagger}(V(\eta))$ is stable under the actions
of $\varphi_{q}$ and $\Gamma_{K}$. As $\B_{\rig,K}^{\dagger}\otimes_{K}E$
is a Bézout domain and $I$ is finitely generated, it is principal,
and we conclude from Lemma 1.1 of \cite{FX13} that $I=\left(t_{K}^{n}\right)$
for $n\in\Z_{\geq0}$. Proposition 6.6 shows that $n=0$, and this
means that 
\[
\left(\B_{\rig,K}^{\dagger}\otimes_{K}E\right)\cdot f(\eta^{-1})\cong\left(\B_{\rig,K}^{\dagger}\otimes_{K}E\right)\left(\eta^{-1}\mu_{\alpha}\right)
\]
 is a rank 1 saturated submodule of $\D_{\rig,K}^{\dagger}(V)$. We
then have 
\[
\D_{\rig,K}^{\dagger}(V)/\left(\B_{\rig,K}^{\dagger}\otimes_{K}E\right)\left(\eta^{-1}\mu_{\alpha}\right)\cong\left(\B_{\rig,K}^{\dagger}\otimes_{K}E\right)\left(\eta\mu_{\alpha^{-1}}\det(V)\right)
\]
 by the classification of $\left(\varphi_{q},\Gamma_{K}\right)$-modules
of rank 1.
\end{proof}
Finally, we conclude with the proof of Theorem B from the introduction.
To do so, we first recall what are cyclotomic trianguline representations.
Let $K_{\infty}^{\cyc}=K(\mu_{p^{\infty}})$ be the cyclotomic extension
of $K$ and let $K_{0}'$ be the maximal unramified extension of $K_{0}$
in $K_{\infty}^{\cyc}$. The ring $\B_{\rig,K}^{\dagger,\mathrm{cyc}}$
is the ring of power series $\sum_{n\in\Z}a_{n}T^{n}$ with $a_{n}\in K_{0}'$
and such that $f(T)$ converges on some nonempty annulus $r<|T|<1$.
The ring is endowed with a $\Frob_{p}$-semilinear $\varphi$ action
and a semilinear $\Gamma_{K}^{\mathrm{cyc}}$-action. If $K=K_{0}$
then $\varphi(T)=(1+T)^{p}-1$ and $\gamma(T)=(1+T)^{\chi_{\mathrm{cyc}}(\gamma)}-1$,
but in general the action has to do with the theory of lifting the
field of norms and is more complicated.

We can then define a notion of a $\left(\varphi,\Gamma_{K}^{\mathrm{cyc}}\right)$-module
over $\B_{\rig,K}^{\dagger,\mathrm{cyc}}\otimes_{\Q_{p}}E$ analogous
to the notion of a $(\varphi,\Gamma_{K})$-module over $\B_{\rig,K}^{\dagger}\otimes_{\Q_{p}}E$.
If $V$ is an $E$-linear representation of $G_{K}$, one can associate
to $V$ a $\left(\varphi,\Gamma_{K}^{\mathrm{cyc}}\right)$-module
$\D_{\rig,K}^{\dagger,\cyc}(V)$ over $\B_{\rig,K}^{\dagger,\mathrm{cyc}}\otimes_{\Q_{p}}E$.
Now let $\delta:K^{\times}\rightarrow E^{\times}$ a continuous character;
we can define a $\left(\varphi,\Gamma_{K}^{\mathrm{cyc}}\right)$-module
$\left(\B_{\rig,K}^{\dagger,\mathrm{cyc}}\otimes_{\Q_{p}}E\right)\left(\delta\right)$
in the following way. If $\delta$ is unitary, then it corresponds
to a character $\delta:G_{K}\rightarrow E^{\times}$, and we set $\left(\B_{\rig,K}^{\dagger,\mathrm{cyc}}\otimes_{\Q_{p}}E\right)\left(\delta\right)=\D_{\rig,K}^{\dagger,\cyc}(V)(E(\delta))$.
If $\delta|_{\mathcal{O}_{K}^{\times}}=1$, set 
\[
\left(\B_{\rig,K}^{\dagger,\mathrm{cyc}}\otimes_{\Q_{p}}E\right)\left(\delta\right)=\left(\B_{\rig,K}^{\dagger,\mathrm{cyc}}\otimes_{\Q_{p}}E\right)\left[\varphi\right]\otimes_{\left(\B_{\rig,K}^{\dagger,\mathrm{cyc}}\otimes_{\Q_{p}}E\right)\left[\varphi_{q}\right]}Ee_{\delta},
\]
where $\varphi_{q}(e_{\delta})=\delta(\pi)e_{\delta}$. For general
$\delta$, write $\delta=\delta_{1}\delta_{2}$ where $\delta_{1}$
is unitary and $\delta_{2}|_{\mathcal{O}_{K}^{\times}}$ and set $\left(\B_{\rig,K}^{\dagger,\mathrm{cyc}}\otimes_{\Q_{p}}E\right)\left(\delta\right)=\left(\B_{\rig,K}^{\dagger,\mathrm{cyc}}\otimes_{\Q_{p}}E\right)\left(\delta_{1}\right)\otimes\left(\B_{\rig,K}^{\dagger,\mathrm{cyc}}\otimes_{\Q_{p}}E\right)\left(\delta_{2}\right)$.

An an $E$-linear representation $V$ of $G_{K}$ is said to be cyclotomic
trianguline if $\D_{\rig,K}^{\dagger,\cyc}(V)$ is a successive extension
$\left(\varphi,\Gamma_{K}^{\cyc}\right)$-modules of the form $\left(\B_{\rig,K}^{\dagger,\mathrm{cyc}}\otimes_{\Q_{p}}E\right)\left(\delta\right)$.
This is the same notion of triangulinity which appears in \cite{Na09,KPX14,Li12},
but we give it a different name here to distinguish it from Lubin-Tate
triangulinity.
\begin{thm}
Let $V$ be a 2-dimensional $E$-linear $K$-analytic representation
of $G_{K}$. The following are equivalent.
\end{thm}

\begin{enumerate}
\item \emph{$V$ is cyclotomic trianguline.}
\item \emph{There exists a $K$-analytic character $\eta:\mathcal{O}_{K}^{\times}\rightarrow E^{\times}$
and $\alpha\in E^{\times}$ such that $\D_{\cris,\Q_{p}}(V(\eta))^{\varphi_{q}=\alpha}$
is nonzero.}
\item \emph{There exists a $K$-analytic character $\eta:\mathcal{O}_{K}^{\times}\rightarrow E^{\times}$
and $\alpha\in E^{\times}$ such that $\D_{\cris,K}(V(\eta))^{\varphi_{q}=\alpha}$
is nonzero.}
\item $V$ \emph{is Lubin-Tate trianguline.}
\end{enumerate}
\begin{proof}
The equivalence between 3 and 4 was proven in Proposition 6.7, while
the equivalence between 2 and 3 follows from Lemma 3.11. It remains
to prove the equivalence of 1 and 2. This equivalence seems to be
well known but due to a lack of suitable reference when $K\neq\Q_{p}$
we give a proof here.

If $V$ is cyclotomic trianguline, then $\D_{\rig,K}^{\dagger,\cyc}(V)$
can be written as an extension 
\[
0\rightarrow\left(\B_{\rig,K}^{\dagger,\mathrm{cyc}}\otimes_{\Q_{p}}E\right)\left(\delta_{1}\right)\rightarrow\D_{\rig,K}^{\dagger,\cyc}(V)\rightarrow\left(\B_{\rig,K}^{\dagger,\mathrm{cyc}}\otimes_{\Q_{p}}E\right)\left(\delta_{2}\right)\rightarrow0.
\]
Since $V$ is $K$-analytic, $\delta_{1}$ is also $K$-analytic.
Twisting by $\delta_{1}|_{\mathcal{O}_{K}^{\times}}^{-1}$, we may
assume $\delta_{1}|_{\mathcal{O}_{K}^{\times}}=1$. It then follows
from \cite[Example 6.2.6]{KPX14} that 
\[
\D_{\cris}\left(\left(\B_{\rig,K}^{\dagger,\mathrm{cyc}}\otimes_{\Q_{p}}E\right)\left(\delta_{1}\right)\right)=\mathrm{I}_{\Q_{p}}^{K}\left(Ee_{\delta_{1}}\right),
\]
where $\varphi_{q}(e_{\delta_{1}})=\delta_{1}(\pi)e_{\delta_{1}}$.
It follows that $\D_{\cris,\Q_{p}}(V)^{\varphi_{q}=\delta_{1}(\pi)}\neq0$.

Conversely, suppose that 2 holds. By replacing $V$ with a $K$-analytic
twist, we may assume that $\D_{\cris,\Q_{p}}^{+}(V)^{\varphi_{q}=\alpha}=\D_{\cris,\Q_{p}}(V)^{\varphi_{q}=\alpha}\neq0$.
It follows from Berger's dictionary that
\[
\D_{\rig,K}^{\dagger,\cyc}(V)^{\Gamma_{K}^{\cyc},\varphi_{q}=\alpha}=\D_{\cris,\Q_{p}}^{+}(V)^{\varphi_{q}\neq0},
\]
so that $\D_{\rig,K}^{\dagger,\cyc}(V)^{\Gamma_{K}^{\cyc},\varphi_{q}=\alpha}$
contains a $(\varphi_{q},\Gamma_{K}^{\cyc})$ invariant $E$-line,
and hence $\text{\ensuremath{\left(\B_{\rig,K}^{\dagger,\mathrm{cyc}}\otimes_{\Q_{p}}E\right)\left(\delta\right)}}$
where $\delta|_{\mathcal{O}_{K}^{\times}}=1$ and $\delta(\pi_{K})=\alpha$.
This sub $\B_{\rig,K}^{\dagger,\mathrm{cyc}}\otimes_{\Q_{p}}E$-module
may not be saturated, but it follows from \cite[Corollary 6.2.9]{KPX14}
that $\D_{\rig,K}^{\dagger,\cyc}(V)$ contains a saturated module
of the form $\left(\B_{\rig,K}^{\dagger,\mathrm{cyc}}\otimes_{\Q_{p}}E\right)\left(\delta'\right)$.
In particular, $\D_{\rig,K}^{\dagger,\cyc}(V)$ is an extension of
two rank 1 $(\varphi_{q},\Gamma_{K}^{\cyc})$-modules, so $V$ is
cyclotomic trianguline.
\end{proof}

\section{Overconvergent Hilbert modular forms}

\subsection{Overconvergent Hilbert eigenforms}

We briefly recall what we need about the cuspidal Hilbert eigenvariety
of Andreatta, Iovita and Pilloni (see \cite{AIP16}). 

Let $F$ be a totally real number field, $\Sigma$ the set of embeddings
of $F$ in $\overline{\Q}$ and $N\in\Z_{\geq4}$. A choice of an
embedding $\overline{\Q}\hookrightarrow\overline{\Q}_{p}$ determines
a decomposition $\Sigma=\amalg_{v:v|p}\Sigma_{F_{v}}$ where each
$v$ is a place of $F$ lying over $p$. Let $L$ be a finite extension
of $\Q_{p}$ which contains $F^{\Gal}$. The weight space for the
algebraic group $\Res_{\mathcal{O}_{F}/\Z}\GL_{2}$ is $\mathcal{W}=\mathrm{Spf}\left(\mathcal{O}_{L}\left[\left[\left(\mathcal{O}_{F}\otimes_{\Z}\Z_{p}\right)^{\times}\times\Z_{p}^{\times}\right]\right]\right)^{\rig}$.
If $f$ is a classical Hilbert eigenform on $F$ of tame level $N$,
its weight is a tuple $\mathrm{wt}(f)=\left(\left\{ k_{\tau}\right\} _{\tau\in\Sigma},w\right)\in\Z_{\geq1}^{\Sigma}\times\Z$
satisfying $k_{\tau}\equiv w\mod2$ for each $\tau\in\Sigma$. It
is then identified with the point in $\mathcal{W}$ corresponding
to the character $(z_{1},z_{2})\mapsto\left(\prod_{\tau\in\Sigma}\tau(z_{1})^{k_{\tau}}\right)z_{2}^{w}$.
The cuspidal Hilbert eigenvariety of tame level $N$ is a certain
rigid analytic space $\mathcal{E}$ which gives a $p$-adic interpolation
of classical Hilbert eigenforms. More precisely, it is a rigid analytic
space together with a weight map $\mathrm{wt}:\mathcal{E}\rightarrow\mathcal{W}$
whose points parametrize overconvergent Hilbert modular forms of finite
slope together with a choice of Hecke eigenvalues at places $v|p$.
We summarize its properties below (see $\mathsection5$ of \cite{AIP16}).
\begin{thm}
1. The map $\mathrm{wt}:\mathcal{E}\rightarrow\mathcal{W}$ is, locally
on $\mathcal{E}$ and $\mathcal{W}$, finite and surjective.

2. For each $\kappa\in\mathcal{W}(\C_{p})$, the fiber $\mathrm{wt}^{-1}(\mathcal{\kappa})$
is in bijection with finite slope Hecke eigenvalues appearing in the
space of overconvergent cusp forms of weight $\kappa$, level $N$
and coefficients in $\C_{p}$.

3. There exists a universal Hecke character $\lambda:\mathcal{H}^{Np}\otimes\mathcal{U}_{p}\rightarrow\mathcal{O}_{\mathcal{E}}$.
Here, $\mathcal{H}^{Np}$ is the abstract Hecke algebra away from
$Np$, and $\mathcal{U}_{p}$ is the $\Q_{p}$-algebra generated by
the $U_{v}$-operators for $v|p$.

4. There is a universal pseudo-character $T:\Gal(\overline{F}/F)\rightarrow\mathcal{O}_{\mathcal{E}}$
which is unramified for $\mathfrak{l}\nmid Np$ such that $T(\Frob_{\mathfrak{l}})=\lambda\left(T_{\mathfrak{l}}\right)$
for the arithmetic Frobenius $\Frob_{\mathfrak{l}}$.

5. For each $x\in\mathcal{E}$ there exists a semisimple Galois representation
$\rho_{x}:\Gal(\overline{F}/F)\rightarrow\GL_{2}\left(k(x)\right)$
which is unramified for $\mathfrak{l}\nmid Np$ and which is characterized
by $\Tr(\rho_{x})=T_{x}$ and $\det(\rho_{x})=\mathrm{Nm}_{F/\Q}\left(\mathfrak{l}\right)\lambda_{x}(S_{\mathfrak{l}})$.

6. The generalized Hodge-Tate weights of $\rho_{x}|_{G_{F_{v}}}$
are $\left\{ \frac{w-k_{\tau}}{2},\frac{w+k_{\tau}-2}{2}\right\} _{\tau\in\Sigma_{F_{v}}}$.
\end{thm}

We fix a place $v|p$ in $F$ and place ourselves in the setting of
$\mathsection1.2$ with $K=F_{v}$, $\pi=\pi_{v}$ a uniformizer of
$F_{v}$, etc.  We extend scalars if necessary so that $\rho_{x}|_{G_{F_{v}}}$
is $F_{v}^{\Gal}$-linear.
\begin{prop}
For $x\in\mathcal{E}$, we have
\[
\D_{\cris,F_{v}}^{+}\left(\rho_{x}^{\vee}|_{G_{F_{v}}}\left(\prod_{\tau\in\Sigma_{F_{v}}}\left(\tau\circ\chi_{\pi_{v}}\right)^{\frac{w-k_{\tau}}{2}}\right)\right)^{\varphi_{q}=\prod_{\tau\in\Sigma_{F_{v}}}\tau(\pi_{v})^{\frac{k_{\tau}-w}{2}}U_{v}}\neq0.
\]
\end{prop}

\begin{proof}
For classical Hilbert modular forms of cohomological weights this
is known by Saito's local-global compatibility results in \cite{Sa09}.
The regular classical points are Zariski dense in $\mathcal{E}$ by
the classicality criterion in \cite{Bi16}, so the claim follows from
the global triangulation results in Theorem 6.3.13 of \cite{KPX14}
or in Theorem 4.4.2 of \cite{Li12}.

For example, in the notation of \cite{Li12}, Theorem 7.1, Saito's
results and the classicality criterion in \cite{Bi16} imply that
$\rho_{x}^{\vee}|_{G_{F_{v}}}$ is a family of refined $p$-adic representations
of $G_{F_{v}}$ of dimension 2, where $\kappa_{1}=\left\{ \frac{k_{\tau}-w}{2}\right\} _{\tau}$,
$\kappa_{2}=\left\{ \frac{-k_{\tau}-w}{2}+1\right\} _{\tau}$, $F_{1}=\prod_{\tau}\tau(\pi_{v})^{\frac{k_{\tau}-w}{2}}U_{v}$,
$\eta_{1}=\prod_{\tau\in\Sigma_{F_{v}}}\left(\tau\circ\chi_{\pi_{v}}\right)^{\frac{k_{\tau}-w}{2}}$
and $Z$ is the set of crystalline points of $\mathcal{E}$. Then
Theorem 4.4.2 of \cite{Li12} (or rather, its proof) establishes the
desired result after applying Berger's dictionary to pass from $\D_{\rig,F_{v}}^{\dagger}$
to $\D_{\cris,F_{v}}^{+}$.
\end{proof}

\subsection{Lubin-Tate triangulation}

Let $x\in\mathcal{E}$ and and consider $\rho_{x}|_{G_{F_{v}}}$ as
an $E$-linear representation for some finite extension $\Q_{p}\subset E$
which contains $F_{v}^{\Gal}$ and $\overline{k}(x)$. In this section
we shall describe exactly when $\rho_{x}|_{G_{F_{v}}}$ is either
overconvergent or Lubin-Tate trianguline and explicitly describe the
triangulations. We shall assume $\rho_{x}|_{G_{F_{v}}}$ is nonsplit.
The split case is less interesting because in that case it is obvious
that $\rho_{x}|_{G_{F_{v}}}$ is overconvergent and that Lubin-Tate
triangulations exist in a degenerate sense. 

Corollary 5.3 gives the following.
\begin{prop}
$\rho_{x}|_{G_{F_{v}}}$ is overconvergent if and only it is $F_{v}$-analytic
up to a twist.
\end{prop}

Let us assume then that $\rho_{x}|_{G_{F_{v}}}$ is s $F_{v}$-analytic
up to a twist, so that that the weights at $\Sigma_{F_{v}}$ are $(k,1,...,1)$
where $k=k_{\Id}$. Let $a_{v}$ be eigenvalue of $U_{v}$ for the
corresponding Hecke operator of $v$. Then $\alpha_{v}=\pi_{v}^{\frac{k-1}{2}}\left(\N_{F_{v}/\Q_{p}}(\pi_{v})\right)^{\frac{1-w}{2}}a_{v}$
interpolates to a function on $\mathcal{E}$ (see Remark 4.7 of \cite{AIP16}).
Upon writing
\[
V=\rho_{x}^{\vee}|_{G_{F_{v}}}\left(\chi_{\pi_{v}}^{\frac{1-k}{2}}\left(\N_{F_{v}/\Q_{p}}\circ\chi_{\pi_{v}}\right)^{\frac{w-1}{2}}\right),
\]
Proposition 7.2 becomes the statement $\D_{\cris,F_{v}}^{+}(V)^{\varphi_{q}=\alpha_{v}}\neq0$.
The representations $\rho_{x}|_{G_{F_{v}}}$ and $V$ differ by a
dual and a character twist, so according to Corollary 5.3 their overconvergence
and Lubin-Tate triangulinity are equivalent. However, $V$ is $F_{v}$-analytic
with $F_{v}$-Hodge-Tate weights $0$ and $k-1$ which makes it nicer
to work with. 

The following is a generalization of Proposition 5.2 of \cite{Ch08}. 
\begin{thm}
The representation $V$ is Lubin-Tate trianguline. We have $\D_{\rig,F_{v}}^{\dagger}\left(V\right)=\D_{\rig,F_{v}}^{\dagger}(s)$
for $s=\left(\delta_{1},\delta_{1}^{-1}\det(V),\mathcal{L}\right)\in\mathscr{S}_{+}^{\an}$,
where
\end{thm}

\begin{enumerate}
\item If $k\notin\Z_{\geq1}$ then $\delta_{1}=\mu_{\alpha_{v}}$, $\mathcal{L}=\infty$
and $s\in\mathscr{S}_{+}^{\mathrm{ng}}$.
\item If $k\in\Z_{\geq1}$ and $\val_{\pi_{v}}(\alpha_{v})<k-1$ then $\delta_{1}=\mu_{\alpha_{v}}$
and either 
\begin{enumerate}
\item $\mathcal{L}=\infty$, in which case $s\in\mathscr{S}_{+}^{\cris}$. 
\item $\mathcal{L}\neq\infty$, in which case $s\in\mathscr{S}_{+}^{\st}$.
This is only possible if $2\val_{\pi_{v}}(\alpha_{v})+\left[F_{v}:\Q_{p}\right]=k-1$.
\end{enumerate}
\item If $k\in\Z_{>1}$ and $\val_{\pi_{v}}(\alpha_{v})=k-1$, then $\delta_{1}=\mu_{\alpha_{v}}$,
$\mathcal{L}=\infty$ and $s\in\mathscr{S}_{+}^{\mathrm{ord}}$.
\item If $k\in\Z_{\geq1}$ and $\val_{\pi_{v}}(\alpha_{v})>k-1$ then $\delta_{1}=x^{1-k}\mu_{\alpha_{v}}$,
$\mathcal{L}=\infty$ and $s\in\mathscr{S}_{+}^{\mathrm{ng}}$.
\end{enumerate}
\begin{proof}
By Proposition 6.7, we know $V$ is Lubin-Tate trianguline and a triangulation
is determined the by largest $i\in\Z$ with $\Fil^{i}\D_{\cris,F_{v}}(V)^{\varphi_{q}=\alpha_{v}}\nsubseteq\Fil^{i+1}\D_{\cris,F_{v}}(V)^{\varphi_{q}=\alpha_{v}}$.
It remains to determine $i$ in each case; it is a nonnegative $F_{v}$-Hodge-Tate
weight of $V$. If $k\notin\Z_{>1}$ then $i=0$, so (1) is settled
and we may assume $k\in\Z_{>1}$.  

Assume that $\val_{\pi_{v}}(\alpha_{v})<k-1$ and suppose by contradiction
that $i=k-1$. Then $\D_{\rig,F_{v}}^{\dagger}\left(V\right)$ has
$\left(\B_{\rig,F_{v}}^{\dagger}\otimes_{F_{v}}E\right)(x^{1-k}\mu_{\alpha_{v}})$
as a subobject, and the latter has slope $\val_{\pi_{v}}(\alpha_{v})-(k-1)<0$
which contradicts Kedlaya's slope filtration theorem (theorem 6.10
of \cite{Ke04}). Thus $i=0$. For the equality in part (b) of (2),
observe that $\mathcal{L}\neq\infty$ can only occur if $\dim_{E}\H_{\an}^{1}\left(\delta_{1}\delta_{2}^{-1}\right)>1$,
which by Theorem 6.2 implies $\delta_{1}\delta_{2}^{-1}=\mu_{q^{-1}}x^{k-1}$.
This proves (2).

For (3), suppose by contradiction that $i=k-1$. Then $\delta_{1}=x^{1-k}\mu_{\alpha_{v}}$
and $s\in\mathscr{S}_{0}^{\mathrm{cris}}\amalg\mathscr{S}_{0}^{\st}$,
so by Corollary 6.4 we have that $V$ is de Rham. A similar argument
to Lemma 6.7 of \cite{Ki03} shows that $V$ must be split, contradicting
our assumption that $\rho_{x}|_{G_{F_{v}}}$ is nonsplit.

Finally, suppose that $\val_{\pi_{v}}(\alpha_{v})>k-1$ and suppose
by contradiction that $i=0$. Then $\D_{\rig,F_{v}}^{\dagger}\left(V\right)$
is an extension of $\left(\B_{\rig,F_{v}}^{\dagger}\otimes_{F_{v}}E\right)(\delta_{1})$
by $\left(\B_{\rig,F_{v}}^{\dagger}\otimes_{F_{v}}E\right)(\delta_{2})$
with $w(\delta_{1})=0$ and $w(\delta_{2})=1-k$. This implies by
Corollary 6.4 that $V$ is $F_{v}$-de Rham, and hence also $F_{v}$-potentially
semistable by Corollary 3.10. But this contradicts admissiblity because
$\val_{\pi_{v}}(\alpha_{v})>k-1$.
\end{proof}
\begin{rem}
$\ $
\end{rem}

\begin{enumerate}
\item If $k,w\in\Z$ then $\val_{\pi_{v}}(\alpha_{v})=\frac{k-1}{2}+\frac{w-1}{2}\left[F_{v}:\Q_{p}\right]+\val_{\pi_{v}}(a_{v})$.
The small slope condition $0\leq\val_{\pi_{v}}(\alpha_{v})\leq k-1$
can then be rewritten as
\[
\frac{1-k}{2}+\frac{w-1}{2}\left[F_{v}:\Q_{p}\right]\leq\val_{\pi_{v}}(a_{v})\leq\frac{k-1}{2}+\frac{w-1}{2}\left[F_{v}:\Q_{p}\right].
\]
\item The parameter $\mathcal{L}\neq0$ appearing in the case 2(b) is described
in the work of Ding (Corollary 2.3 of \cite{Di17}) in the following
way. Upon considering these points of $\mathcal{E}$ with weights
$(\kappa,1,...,1)$ in a small affinoid neighborhood of $x$, one
has
\[
\mathcal{L}(x)=-2\frac{d\log\alpha_{v}}{d\kappa}|_{\kappa=k}.
\]
\item When $F=\Q$ and $k\geq2$, Coleman's classicality theorem (\cite{Co97})
says that $f$ is classical if and only if $\val_{p}(a_{p})<k-1$
or $\val_{p}(a_{p})=k-1$ and $f$ is not in the image of $\Theta^{k-1}$,
where $\Theta$ is the operator which acts on $q$-expansions by $q\frac{d}{dq}$.
Analogously, we can give a prediction in general when $p$ is an inert
prime in $F$. We expect, based on the conjectures appearing in $\mathsection4$
of \cite{Br10}, that an $F_{p}$-analytic form $f$ is classical
if and only if $\val_{p}(\alpha_{p})<k-1$ or $\val_{p}(\alpha_{p})=k-1$
and $f$ is not in the image of $\Theta_{\Id}^{k-1}$. Here $\Theta_{\mathrm{Id}}$
is the Theta operator in the direction of the identity embedding,
as constructed in $\mathsection15$ of \cite{AG05}. If such a classicality
statement were known, one could argue as in $\mathsection6$ of \cite{Ki03}
and deduce the Fontaine-Mazur conjecture for the representations attached
to $F_{p}$-analytic finite slope Hilbert eigenforms.
\item If we allow $\rho_{f}|_{G_{F_{v}}}$ to be split, it is also possible
that $k\in\Z_{\geq1}$, $\val_{\pi_{v}}(\alpha_{v})=k-1$ and $\Fil^{k-1}\D_{\cris,F_{v}}(V)^{\varphi_{q}=\alpha_{v}}\neq0$.
Our expectation is that if $f$ itself is not classical then $f=\Theta_{\Id}^{k-1}g$
for some eigenform $g$, so that this is the only case where $\rho_{f}|_{G_{F_{v}}}$
can be de Rham without $f$ itself being classical. In the case of
$F=\Q$, this is known by $\mathsection6$ of \cite{Ki03}.
\end{enumerate}

\subsection{Example: the eigenform of Moy and Specter}

In this section we shall test our results for a classical Hilbert
eigenform. It is not too easy to find \emph{explicit} classical Hilbert
eigenforms for which Theorem 7.4 gives any new information beyond
that which already exists in the literature. The case where $v$ splits
in $F$ is well understood, and for CM Hilbert eigenforms of $F$
the local representation at $F_{v}$ splits so Theorem 7.4 is rather
trivial. That's why we shall consider in this subsection the non-CM
Hilbert eigenform of partial weight 1 found by Moy and Specter in
\cite{MS15}. To the best of the author's knowledge, it is the only
example in the literature of a non-CM classical Hilbert eigenform
of partial weight 1. 

Recall that if $f$ is a classical Hilbert eigenform of level $\Gamma_{1}(N)$
and nebentypus $\varepsilon$, the Hecke polynomial $P_{v}(X)$ at
a place $v$ with $a_{v}\neq0$ is given by
\[
P_{v}(X)=\begin{cases}
X-a_{v} & \text{if }v\mid N\\
X^{2}-c(v,f)X+\varepsilon(v)\N_{F/\Q}(v)^{w-1} & \text{if }v\nmid N
\end{cases},
\]
where $c(v,f)$ is the $T_{v}$-eigenvalue. When $v\nmid N$, raising
the level of $f$ gives two eigenforms $f_{1},f_{2}$ whose attached
$p$-adic representations $\rho_{f}$ coincide and such that $\left\{ a_{v}(f_{1}),a_{v}(f_{2})\right\} $
are the two roots of $P_{v}(X)$. Using Theorem 7.4, this gives rise
to two different triangulations of $V=\rho_{x}^{\vee}|_{G_{F_{v}}}\left(\chi_{\pi_{v}}^{\frac{1-k}{2}}\left(\N_{F_{v}/\Q_{p}}\circ\chi_{\pi_{v}}\right)^{\frac{w-1}{2}}\right)$.
Whenever local-global compatibility holds, the Hecke polynomial is
equal to the characteristic polynomial of the action of $\varphi_{q}$
on $\D_{\cris}^{+}\left(\rho_{f}|_{G_{F_{v}}}^{\vee}\right)$. Thus
the valuation of $c(v,f)$ determines the valuations of the eigenvalues
of $\varphi_{q}$ by the method of the Newton polygon. This observation
is used in the computations below.

Next recall that the main theorem of \cite{MS15} finds for $F=\Q(\sqrt{5})$
a non CM cuspidal Hilbert eigenform $f$ of weights $(k_{1},k_{2},w)=(5,1,5)$,
level $\Gamma_{1}(14)$, nebentypus $\varepsilon$ with conductor
$7\left(\infty_{1}\right)\left(\infty_{2}\right)$. For the following
examples, we let $p$ be a prime in the range $[2,11]$, $v$ a place
of $F$ lying over that prime and $\rho_{f}$ the associated $p$-adic
Galois representation of $f$. We set
\[
V=\rho_{x}^{\vee}|_{G_{F_{v}}}\left(\chi_{\pi_{v}}^{-2}\left(\N_{F_{v}/\Q_{p}}\circ\chi_{\pi_{v}}\right)^{2}\right),
\]
which differs from $\rho_{f}|_{G_{F_{v}}}$ only by a dual and a crystalline
twist. We shall examine the behaviour of $V$ for different $v$.
When $v\neq(2)$ local-global compatibility holds by Remark 1.5 of
\cite{Ne15}, while in $v=(2)$ we shall assume it holds, though it
seems to be still conjectural in this case. Local-global compatibility
implies that $\rho_{f}|_{G_{F_{v}}}$ is de-Rham, and since its Hodge-Tate
weights at each nontrivial embedding of $F_{v}$ are $\{0,0\}$, it
is also $F_{v}$-analytic. Given an eigenvalue $a_{v}$ of $U_{v}$,
Theorem 7.4 produces a point $s\in\mathscr{S}_{+}^{\an}$. The table
in $\mathsection3$ of \cite{MS15} computes the values of $a_{v}$
for such $v$. It has to lie in the range given by Remark 7.5(1).

\textbf{Examples.}

1. The place $v=\left(2\right)$ lies over the inert prime $p=2$
and the valuation bound is $\val_{v}(a_{v})\in\left[2,6\right]$.
Since the character has conductor prime to $2$ and the level at $2$
is $\Gamma_{0}(2)$, the local component $\pi_{2}(f)$ is Steinberg
(up to an unramified quadratic twist). A suitable local-global compatibility
theorem predicts that $V$ is semistable noncrystalline and $s\in\mathscr{S}_{+}^{\st}$.
In particular, the condition of case 2(b) of Theorem 7.4 predicts
that $\val_{2}(a_{2})=3$, which is confirmed by $\mathsection3$
of \cite{MS15}.

2. The place $v=\left(3\right)$ lies over the inert prime $p=3$
and the valuation bound is $\val_{v}(a_{v})\in\left[2,6\right]$.
The place $v$ is coprime to the level, so the local component $\pi_{3}(f)$
is unramified principal series. By local-global compatibility, $V$
is crystalline. By $\mathsection3$ of \cite{MS15}, we have $\val_{3}\left(c(3,f)\right)=2$,
so that the two $U_{v}$-eigenvalues have valuations $2$ and $4$.
Then $V$ has two triangulations, giving rise to $s_{1}\in\mathscr{S}_{0}^{\cris}$
and $s_{2}\in\mathscr{S}_{+}^{\mathrm{ord}}$.

3. The place $v=\left(\sqrt{5}\right)$ lies over the ramified prime
$p=5$ and the valuation bound is $\val_{v}(a_{v})\in\left[2,6\right]$.
By $\mathsection3$ of \cite{MS15}, we have $\val_{v}\left(c(v,f)\right)=2$,
and the triangulations in this case behave similar to the case of
$v=(3)$.

4. The place $v=\left(7\right)$ lies over the inert prime $p=7$
and the valuation bound is $\val_{v}(a_{v})\in\left[2,6\right]$.
The character has conductor divisible by $7$ and the level at $7$
is $\Gamma_{0}(7)$, so the local component $\pi_{7}(f)$ is ramified
prinicpal series. After an abelian extension it becomes unramified
principal series, so by local-global compatibility $V$ is crystabelline.
By $\mathsection3$ of \cite{MS15}, we have $\val_{7}\left(a_{7}\right)=3$,
so $V$ gives rise to $s\in\mathscr{S}_{+}^{\cris}\backslash\mathscr{S}_{0}^{\cris}$.

5. The place $v=\left(\frac{7+\sqrt{5}}{2}\right)$ lies over the
split prime $p=11$ and the valuation bound is $\val_{v}(a_{v})\in\left[0,4\right]$.
By $\mathsection3$ of \cite{MS15}, we have $\val_{v}\left(c(v,f)\right)=0$,
and the triangulations in this case behave similar to the case of
$v=(3)$ and $v=\left(\sqrt{5}\right)$.

.

\end{document}